\documentclass[10pt]{amsart}

\usepackage{amssymb}
\usepackage{amsmath}
\usepackage[all]{xy}
\usepackage{enumerate}

\newtheorem{theorem}[subsection]{Theorem}
\newtheorem*{ntheorem}{Theorem}
\newtheorem{cor}[subsection]{Corollary}

\newtheorem{lemma}[subsection]{Lemma}

\theoremstyle{remark}
  
\newtheorem{rem}[subsection]{Remark}

\theoremstyle{definition}

\DeclareMathOperator{\ord}{ord}
\DeclareMathOperator{\codim}{codim}
\DeclareMathOperator{\BOX}{Box}
\DeclareMathOperator{\Hom}{Hom}
\DeclareMathOperator{\Ext}{Ext}
\DeclareMathOperator{\Spec}{Spec}
\DeclareMathOperator{\Aut}{Aut}

\DeclareMathOperator{\Pic}{Pic}

\DeclareMathOperator{\sft}{sft}
\DeclareMathOperator{\age}{age}

\newcommand{\bSigma}{\mbox{\boldmath$\Sigma$}}
\newcommand{\btau}{\mbox{\boldmath$\tau$}}
\newcommand{\bsigma}{\mbox{\boldmath$\sigma$}}

\newcommand{\bsv}{\mbox{\boldmath$\sigma(v)$}}
\newcommand{\bntau}{\mbox{\boldmath$n\tau$}}
\newcommand{\bnSigma}{\mbox{\boldmath$n\Sigma$}}
\newcommand{\btriangle}{\mbox{\boldmath$\triangle$}}

\begin{document}

\title{Motivic Integration on Toric Stacks}

\author{A. Stapledon}

\address{Department of Mathematics, University of Michigan, Ann Arbor, MI 48109, USA}
\email{astapldn@umich.edu}

\begin{abstract}
We present a decomposition of the space of twisted arcs of a toric stack. As a consequence, 
we give a combinatorial description of the motivic integral  associated to a torus-invariant divisor of a toric stack. 
\end{abstract}

\maketitle

\section{Introduction}

Let $N$ be a lattice of rank $d$ and set $N_{\mathbb{R}} = N \otimes_{\mathbb{Z}} \mathbb{R}$. 
Let $\Sigma$ be a simplicial, rational, $d$-dimensional fan in $N_{\mathbb{R}}$, and
assume that the support $|\Sigma|$ of $\Sigma$ is convex. 
Denote the primitive integer generators of $\Sigma$ by $v_{1}, \ldots, v_{r}$ and fix positive integers $a_{1}, \ldots, a_{r}$.
The data $\bSigma = (N, \Sigma, \{ a_{i}v_{i} \} )$ is called a \emph{stacky fan}. In \cite{BCSOrbifold}, Borisov, Chen and Smith  
associated to $\bSigma$ a smooth Deligne-Mumford stack $\mathcal{X} = \mathcal{X}(\bSigma)$ over $\mathbb{C}$, called a \emph{toric stack}, with coarse moduli space 
the toric variety $X = X(\Sigma)$ (see also \cite{FMNSmooth}). 

To any complex variety $Y$, one can associate its corresponding arc space $J_{\infty}(Y) := \Hom( \Spec \mathbb{C}[[t]], Y)$. 
The geometry of $J_{\infty}(Y)$ encodes a lot of information about the birational geometry of $Y$ and has been the subject of intensive study in recent times (see, for example, 
\cite{ELMContact, EMInversion, MusJet}). 
Kontsevich introduced the theory of \emph{motivic integration} in \cite{KonMotivic}, which assigns a measure to subsets of $J_{\infty}(Y)$ and allows one to compare invariants on birationally equivalent varieties.  This theory has been developed over the last decade by many authors including Denef and Loeser \cite{DLMotivic, DLGerms, DLGeometry, DLMotivic2} and 
Batyrev \cite{BatNon,BDStrong}.

Yasuda extended the theory of arc spaces and motivic integration to Deligne-Mumford stacks in \cite{YasTwisted} and \cite{YasMotivic}.
More specifically, to any Deligne-Mumford stack $\mathcal{Y}$, he associated the space $|\mathcal{J}_{\infty}\mathcal{Y}|$ of \emph{twisted arcs} of $\mathcal{Y}$  and defined an associated measure.  
The goal of this paper is to give an explicit description of the space of twisted arcs of the toric stack $\mathcal{X}$ above. 
We will prove a decomposition of $|\mathcal{J}_{\infty}\mathcal{X}|$ 
which allows us to compute motivic integrals on $\mathcal{X}$ and we will show how these motivic integrals relate to combinatorial invariants from the theory of lattice point enumeration of polyhedral complexes. 

We briefly recall Ishii's decomposition of the arc space $J_{\infty}(X)$ of the toric variety $X$ in \cite{IshArc} (for more details, see Section \ref{arc}).
If $T$ denotes the torus of $X$ with corresponding arc space $J_{\infty}(T)$,  
then $J_{\infty}(T)$ has the structure of a group and acts on $J_{\infty}(X)$. If $D_{1}, \ldots, D_{r}$ denote the $T$-invariant divisors of $X$, then the open subset $J_{\infty}(X)' = J_{\infty}(X) \smallsetminus \cup_{i} J_{\infty}(D_{i})$ is invariant under the action of $J_{\infty}(T)$.
Ishii gave a decomposition of $J_{\infty}(X)'$ into $J_{\infty}(T)$-orbits indexed by $|\Sigma| \cap N$ 
and described the orbit closures. 

Our first main result is a stacky analogue of Ishii's decomposition. 
We describe an action of $J_{\infty}(T)$ on 
$|\mathcal{J}_{\infty}\mathcal{X}|$, which restricts to an action on
$|\mathcal{J}_{\infty}\mathcal{X}|' = |\mathcal{J}_{\infty}\mathcal{X}| \smallsetminus  \cup_{i} |\mathcal{J}_{\infty}\mathcal{D}_{i}|$,
where $\mathcal{D}_{1}, \ldots, \mathcal{D}_{r}$ are the $T$-invariant prime divisors of $\mathcal{X}$ (see Section \ref{arc}). 
We define a twisted arc $\tilde{\gamma}_{v}$ for each element $v$ in $|\Sigma| \cap N$,
and describe the $J_{\infty}(T)$-orbits of $|\mathcal{J}_{\infty}\mathcal{X}|'$ (Theorem \ref{decomp}), as follows.

\begin{ntheorem}
We have a decomposition of $|\mathcal{J}_{\infty}\mathcal{X}|'$ into $J_{\infty}(T)$-orbits 
\begin{equation*}
|\mathcal{J}_{\infty}\mathcal{X}|' = \coprod_{v \in |\Sigma| \cap N} \tilde{\gamma}_{v} \cdot J_{\infty}(T).
\end{equation*}
Moreover, $\tilde{\gamma}_{w} \cdot J_{\infty}(T) \subseteq \overline{\tilde{\gamma}_{v} \cdot J_{\infty}(T)}$ if and only if 
$w - v = \sum_{\rho_{i} \subseteq \sigma} \lambda_{i}a_{i}v_{i}$ for some non-negative integers $\lambda_{i}$ and  
some cone $\sigma$ containing $v$ and $w$. 
\end{ntheorem}



We now turn our attention to motivic integrals on $\mathcal{X}$. If $\mathcal{E} = \sum_{i = 1}^{r} b_{i} \mathcal{D}_{i}$ is a $T$-invariant $\mathbb{Q}$-divisor on $\mathcal{X}$, then the pair $(\mathcal{X}, \mathcal{E})$ is \emph{Kawamata log terminal}
if $b_{i} < 1$ for $i = 1, \ldots, r$. To any Kawamata log terminal pair $(\mathcal{X}, \mathcal{E})$, Yasuda associated the motivic integral $\Gamma(\mathcal{X},\mathcal{E})$ of $\mathcal{E}$ on $\mathcal{X}$. In the case when each $a_{i} = 1$, Theorem 3.41 in 
\cite{YasMotivic} implies that $\Gamma(\mathcal{X},\mathcal{E})$ is equal to Batyrev's stringy invariant of a corresponding Kawamata log terminal pair of varieties \cite{BatNon}.
We may view $\Gamma(\mathcal{X},\mathcal{E})$ either as an element of a Grothendieck ring of varieties \cite{YasTwisted} or as an element in a ring of convergent stacks \cite{YasMotivic}. For simplicity, 
we will specialise $\Gamma(\mathcal{X},\mathcal{E})$ and view it as a convergent power series in $\mathbb{Z}[(uv)^{1/N}][[ (uv)^{-1/N} ]]$, for some positive integer $N$, and refer the reader to section \ref{motivic} for the relevant details. 

In order to  motivate and state our results, we recall some notions from the theory of lattice point enumeration of polyhedral complexes (see, for example, \cite{BarLattice}). 
Denote by $\psi: |\Sigma| \rightarrow \mathbb{R}$ the piecewise $\mathbb{Q}$-linear function 
satisfying $\psi(a_{i}v_{i}) = 1$ for $i = 1, \ldots ,r$, and  
consider the polyhedral complex
$Q = \{ v \in |\Sigma| \, | \, \psi(v) \leq 1 \}$. If for each positive integer $m$, $f_{Q}(m)$ denotes the number of lattice points 
in $mQ$, then  $f_{Q}(m)$ is a polynomial in $m$ of degree $d$, called the \emph{Ehrhart polynomial} of $Q$.  
The generating series of $f_{Q}(m)$ can be written in the form
\begin{equation*}
1 + \sum_{m \geq 1} f_{Q}(m)t^{m} = \delta_{Q}(t)/ (1 - t)^{d + 1},
\end{equation*}
where $\delta_{Q}(t)$ is a polynomial of degree less than or equal to $d$,
called the \emph{Ehrhart $\delta$-polynomial} of $Q$. 
More generally, if $\mu: |\Sigma| \rightarrow \mathbb{R}_{\geq 0}$ is a piecewise linear function, one can consider the power series 
\[
\delta_{Q}(t;\mu) := (1 - t)^{d + 1}(1 +  \sum_{m \geq 1} \sum_{v \in mQ \cap N} t^{\mu(v) + m}),
\]
so that $\delta_{Q}(t;\mu) = \delta_{Q}(t)$ when $\mu$ is identically zero.
The Ehrhart $\delta$-polynomial of $Q$ and its associated generalisations have been studied extensively over the last forty years by many authors including Stanley \cite{StaHilbert2, StaDecompositions, StaHilbert1,  StaMonotonicity} and Hibi \cite{HibSome, HibEhrhart,  HibLower, HibStar}.  

Motivated precisely by computations of motivic integrals on toric stacks, the following variant of the above definition was recently introduced in \cite{YoWeightI}.
If  $\lambda: |\Sigma| \rightarrow \mathbb{R}$ is a piecewise $\mathbb{Q}$-linear function satisfying $\lambda(a_{i}v_{i}) > - 1$ for $i = 1, \ldots, r$, then
the \emph{weighted $\delta$-vector} $\delta^{\lambda}(t)$  is defined by
\[
\delta^{\lambda}(t) = (1  - t)^{d + 1}(1 +  \sum_{m \geq 1} \sum_{v \in mQ \cap N} t^{\psi(v) - \lceil \psi(v) \rceil + \lambda(v) + m}).
\]
Observe that if $\lambda$ is a piecewise linear function, then the coefficient of $t^{i}$ in $\delta_{Q}(t; \lambda)$ is equal to the sum of 
the coefficients of $t^{j}$ in $\delta^{\lambda}(t)$ for $i - 1 <  j \leq i$. The weighted $\delta$-vector satisfies the following symmetry property. 

\begin{ntheorem}[\cite{YoWeightI}]
With the notation above, the expression $\delta^{\lambda}(t)$ is a rational function in $\mathbb{Q}(t^{1/N})$, for some positive integer $N$.  
If $\Sigma$ is a complete fan, 
then 
\begin{equation*}
\delta^{\lambda}(t) = t^{d} \delta^{\lambda}(t^{-1}).
\end{equation*}
\end{ntheorem}

The following result (Theorem \ref{tropo}) may be viewed as a geometric realisation of weighted $\delta$-vectors.


\begin{ntheorem}
Consider a Kawamata log terminal pair $(\mathcal{X}, \mathcal{E})$, where $\mathcal{E}$ is a $T$-invariant $\mathbb{Q}$-divisor on the toric stack $\mathcal{X}$. 
There exists a corresponding piecewise $\mathbb{Q}$-linear function $\lambda: |\Sigma| \rightarrow \mathbb{R}$ satisfying $\lambda(a_{i}v_{i}) > - 1$ for $i = 1, \ldots, r$, such that
\begin{equation*}
\Gamma(\mathcal{X}, \mathcal{E}) = (uv)^{d} \delta^{\lambda}((uv)^{-1}).
\end{equation*}
In particular, $\Gamma(\mathcal{X}, \mathcal{E})$  is a rational function in $\mathbb{Q}(t^{1/N})$, for some positive integer $N$.  
If $\Sigma$ is a complete fan, then
\begin{equation*}
\Gamma(\mathcal{X}, \mathcal{E})(u,v) = (uv)^{d}  \Gamma(\mathcal{X}, \mathcal{E})(u^{-1}, v^{-1}) =  \delta^{\lambda}(uv).
\end{equation*}
Moreover, every weighted $\delta$-vector has the form $\delta^{\lambda}(uv) = (uv)^{d}\Gamma(\mathcal{X}, \mathcal{E})(u^{-1}, v^{-1})$, for such a
pair $(\mathcal{X}, \mathcal{E})$.
\end{ntheorem}

In the case when $\mathcal{E} = 0$, the invariant $\Gamma(\mathcal{X}, 0)$ is a polynomial in $\mathbb{Z}[ (uv)^{1/N} ]$, for some positive integer $N$,
and the coefficient of $(uv)^{j}$ is equal to the  dimension of the $2j^{\textrm{th}}$ orbifold cohomology group of $\mathcal{X}$ with compact support  \cite{YasTwisted}. It follows from Poincar\'e duality for orbifold cohomology \cite{CRNew} and the above discussion that 
the coefficients of the Ehrhart $\delta$-polynomial of $Q$ are sums of dimensions of orbifold cohomology groups of $\mathcal{X}$. A detailed discussion of this result is provided in \cite{YoWeightI}. 

We conclude the introduction with an outline of the contents of the paper. 
In Section \ref{tstack}, we recall Borisov, Chen and Smith's notion of a toric stack and give an explicit description of the local construction. We review Yasuda's theory of twisted jets in Section \ref{jet} and establish a concrete description of these spaces in the toric case. We use these results in Section \ref{arc} to prove our decomposition of the space of twisted arcs of a toric stack. 
In Section \ref{contact}, we compute the contact order of a twisted arc along a torus-invariant divisor and, in Section \ref{motivic}, we use the results of the previous sections to compute certain motivic integrals on toric stacks. Finally, in Section \ref{transform}, we interpret Yasuda's change of variables formula in the toric case to give a geometric proof of a combinatorial result in \cite{YoWeightI}. 




\medskip



The author would like to thank his advisor Mircea Musta\c t\v a for all his help.   
He would also like to thank Bill Fulton, Sam Payne and Kevin Tucker for some useful discussions. 
The author was supported by
Mircea Musta\c t\v a's Packard Fellowship and
by an Eleanor Sophia Wood
travelling scholarship from the University of Sydney.

\section{Toric Stacks}\label{tstack}

We fix the following notation throughout the paper. 
Let $N$ be a lattice of rank $d$ and set $N_{\mathbb{R}} = N \otimes_{\mathbb{Z}} \mathbb{R}$. 
Let $\Sigma$ be a simplicial, rational, $d$-dimensional fan in $N_{\mathbb{R}}$, and let $X = X(\Sigma)$ denote the corresponding toric variety with torus $T$. 
We assume that the support $|\Sigma|$ of $\Sigma$ in $N_{\mathbb{R}}$ is convex.  
Let $\rho_{1}, \ldots, \rho_{r}$ denote 
the rays of $\Sigma$, with primitive integer generators $v_{1}, \ldots, v_{r}$ in $N$ and corresponding $T$-invariant divisors $D_{1}, \ldots, D_{r}$. 
Fix elements $b_{1}, \ldots , b_{r}$ in $N$ such that $b_{i} = a_{i}v_{i}$ for some positive integer 
$a_{i}$. 
The data $\bSigma = (N, \Sigma, \{ b_{i} \} )$ is called a \emph{stacky fan}. 
We denote by $\psi: |\Sigma| \rightarrow \mathbb{R}$ the function that is $\mathbb{Q}$-linear on each cone of 
$\Sigma$ and satisfies $\psi(b_{i}) = 1$ for $i = 1, \ldots ,r$. 
For each non-zero cone $\tau$ of $\Sigma$, set 
\begin{equation*}\label{mcbox}
\BOX(\btau) = \{ v \in N \mid  v = \sum_{\rho_{i} \subseteq \tau} q_{i}b_{i} \textrm{  for some  }
0 < q_{i} < 1 \}.
\end{equation*}
We set $\BOX(\mbox{\boldmath$\{ 0 \}$}) = \{0\}$ and $\BOX(\bSigma) = \cup_{\tau \in \Sigma} \BOX(\btau)$.
We will always work over $\mathbb{C}$ and will often identify schemes with their $\mathbb{C}$-valued points. 

Associated to a stacky fan, there is a smooth Deligne-Mumford toric stack $\mathcal{X} = \mathcal{X}(\bSigma)$ over $\mathbb{C}$ with 
coarse moduli space $X$   \cite{BCSOrbifold}. 
Each cone $\sigma$ in $\Sigma$ corresponds to an open substack $\mathcal{X}(\bsigma)$ of $\mathcal{X}$. 
We identify 
$\mathcal{X}(\mathbf{ \{ 0 \}})$ with the torus $T$ of $X$. 
The open substacks $\{ \mathcal{X}(\bsigma) \, | \, \dim \sigma = d \}$ give an open covering of $\mathcal{X}$. 
 We will give an explicit description of $\mathcal{X}(\bsigma)$, for a fixed  cone $\sigma$ of dimension $d$. 

If $N_{\sigma}$ denotes the sublattice of $N$ generated by 
$\{ b_{i} \, | \, \rho_{i} \subseteq \sigma \}$, 
then $N(\sigma) = N/N_{\sigma}$ is a finite group with elements in bijective correspondence with 
$\coprod_{\tau \subseteq \sigma}\BOX(\btau)$. 
Let $M_{\sigma}$ be the dual lattice of $N_{\sigma}$ and let $M$ be the dual lattice of $N$. 
If $\sigma'$ denotes the cone in $N_{\sigma}$ generated by $\{ b_{i} \, | \, \rho_{i} \subseteq \sigma \}$, with corresponding 
dual cone $(\sigma')^{\vee}$ in $M_{\sigma}$, then
we will make the identifications 
\begin{equation*}
\Spec \mathbb{C}[(\sigma')^{\vee} \cap M_{\sigma}] \cong \mathbb{A}^{d}, \: \Hom_{\mathbb{Z}}(M_{\sigma}, \mathbb{C}^{*}) \cong (\mathbb{C}^{*})^{d}.
\end{equation*}
If we regard $\mathbb{Q}/\mathbb{Z}$ as a subgroup of $\mathbb{C}^{*}$ by sending $p$ to $\exp (2 \pi \sqrt{-1}p)$, then we have a natural isomorphism
\begin{equation}\label{eqnZ}
N(\sigma) \cong \Hom_{\mathbb{Z}}( \Hom_{\mathbb{Z}}(N(\sigma), \mathbb{Q}/\mathbb{Z}), \mathbb{C}^{*}) = 
\Hom_{\mathbb{Z}}( \Ext_{\mathbb{Z}}^{1}(N(\sigma), \mathbb{Z}), \mathbb{C}^{*}).
\end{equation}
We apply the functor $\Hom_{\mathbb{Z}}( \, \: , \mathbb{Z})$ to  the exact sequence
\begin{equation*}
0 \rightarrow N_{\sigma} \rightarrow N \rightarrow N(\sigma) \rightarrow 0,
\end{equation*} 
to get
\begin{equation*}
0 \rightarrow M \rightarrow M_{\sigma} \rightarrow \Ext_{\mathbb{Z}}^{1}(N(\sigma),\mathbb{Z}) \rightarrow 0.
\end{equation*}
Applying $\Hom( \, \: , \mathbb{C^{*}})$ and the natural isomorphism  (\ref{eqnZ}), gives an injection
$N(\sigma) \rightarrow (\mathbb{C}^{*})^{d}$.
The natural action of $(\mathbb{C}^{*})^{d}$ on $\mathbb{A}^{d}$ induces an action of $N(\sigma)$ on $\mathbb{A}^{d}$. We identify $\mathcal{X}(\bsigma)$ 
with the quotient stack
$[ \mathbb{A}^{d} / N(\sigma)]$, with  corresponding coarse moduli space the open subscheme $U_{\sigma} = \mathbb{A}^{d} / N(\sigma)$ of $X$ \cite{BCSOrbifold}.

Consider an element $g$ in $N(\sigma)$ of order $l$ 
corresponding to $v$ in $\BOX(\bsv)$, where $\sigma(v)$ denotes the cone containing 
$v$ in its relative interior. We can write
$v = \sum_{i = 1}^{d} q_{i} b_{i}$,
for some $0 \leq q_{i} < 1$. Note that $q_{i} \neq 0$ if and only if $\rho_{i} \subseteq \sigma(v)$. 
If $x_{1}, \ldots, x_{d}$ are the coordinates on $\mathbb{A}^{d} = \Spec \mathbb{C}[(\sigma')^{\vee} \cap M_{\sigma}]$ 
and $\zeta_{l} = \exp(2\pi \sqrt{-1}/l )$, then one can verify that the action of $N(\sigma)$ on $\mathbb{A}^{d}$ is given by
\begin{equation}\label{actions}
g \cdot (x_{1}, \ldots, x_{d}) = (\zeta_{l}^{\, lq_{1}}x_{1}, \ldots , \zeta_{l}^{\,lq_{d}}x_{d}), 
\end{equation}
and the \emph{age} of $g$ (see Subsection 7.1 \cite{AGVAlgebraic}) is equal to
\begin{equation}\label{age}
\age(g) = (1/l) \sum_{i = 1}^{d} lq_{i}  = \psi(v).
\end{equation}

\section{Twisted Jets of Toric Stacks}\label{jet}

We use the discussion  of twisted jets of Deligne-Mumford stacks in  \cite{YasMotivic}
to give an explicit description of the toric case. 
More specifically, for any non-negative integer $n$, we describe the stack of twisted $n$-jets $\mathcal{J}_{n}\mathcal{X}$ of the toric stack $\mathcal{X}$. 


Fix an affine scheme $S = \Spec R$ over $\mathbb{C}$ and
let $D_{n,S}$ denote the affine scheme $\Spec R[t]/(t^{n + 1})$. If we fix a positive integer $l$ 
and consider the group $\mu_{l}$ of $l$th roots of unity  with generator $\zeta_{l} = \exp (2 \pi \sqrt{-1} /l)$,  then $\mu_{l}$ 
acts on $D_{nl,S}$ via the morphism
$p: D_{nl,S} \times \mu_{l} \rightarrow D_{nl,S}$, 
corresponding to the ring homomorphism $R[t]/(t^{nl + 1}) \rightarrow R[t]/(t^{nl + 1}) \otimes \mathbb{C}[x]/(x^{l} - 1), t \mapsto t \otimes x$. That is, $\mu_{l}$
acts on $D_{nl,S}$ by scaling $t$. 
If $\mathcal{D}_{n,S}^{l}$ denotes  the quotient stack $[D_{nl,S} / \mu_{l} ]$, then
we have morphisms
\begin{equation*}
D_{nl,S} \stackrel{\pi}{\rightarrow} \mathcal{D}_{n,S}^{l} \rightarrow D_{n,S},
\end{equation*}
such that $\pi$ is an atlas for  $\mathcal{D}_{n,S}^{l}$ and $D_{n,S}$ is the coarse moduli space of $\mathcal{D}_{n,S}^{l}$. The  composition of the two maps is the quotient of 
$D_{nl,S}$ by $\mu_{l}$ and corresponds to the ring homomorphism $R[t]/(t^{n + 1}) \rightarrow R[t]/(t^{nl + 1})$, $t \mapsto t^{l}$. 
The atlas $\pi$ corresponds to the object $\alpha$ of $\mathcal{D}_{n,S}^{l}$ over $D_{nl,S}$
\begin{equation*}\label{eqnat}
\xymatrix{ D_{nl,S} \times \mu_{l} \ar[d]^{pr_{1}} \ar[r]^p & D_{nl,S} \\
D_{nl,S} &  }
\end{equation*}
and every object in $\mathcal{D}_{n,S}^{l}$ is locally a pullback of $\alpha$. Consider the automorphism 
\begin{equation*}
\theta = \zeta_{l} \times \zeta_{l}^{-1}: D_{nl,S} \times \mu_{l} \rightarrow D_{nl,S} \times \mu_{l}
\end{equation*}
over $\zeta_{l}: D_{nl,S} \rightarrow D_{nl,S}$. Every automorphism of $\alpha$ is a power of $\theta$ and hence every object and automorphism in 
$\mathcal{D}_{n,S}^{l}$ is locally determined
by a pullback of $\alpha$ and a power of $\theta$.

A \emph{twisted $n$-jet of order $l$} of $\mathcal{X}$ over $S$ 
is a representable morphism 
$\mathcal{D}_{n,S}^{l} \rightarrow \mathcal{X}$. 
By the above discussion, a twisted jet is determined by the images of $\alpha$ and $\theta$.
Yasuda defined the stack $\mathcal{J}_{n}^{l}\mathcal{X}$ of twisted $n$-jets of order $l$ of $\mathcal{X}$. 
An object of $\mathcal{J}_{n}^{l}\mathcal{X}$ over $S$ is a twisted $n$-jet $\gamma: \mathcal{D}_{n,S}^{l} \rightarrow \mathcal{X}$ of order $l$. Suppose 
$\gamma': \mathcal{D}_{n,T}^{l} \rightarrow \mathcal{X}$ is another twisted $n$-jet of order $l$. If $f: S \rightarrow T$ is a morphism, there is an induced morphism 
$f': \mathcal{D}_{n,S}^{l} \rightarrow \mathcal{D}_{n,T}^{l}$.
A morphism in $\mathcal{J}_{n}^{l}\mathcal{X}$ from $\gamma$ to $\gamma'$ over 
$f: S \rightarrow T$ is a 2-morphism from $\gamma$ to $\gamma' \circ f'$. The stack $\mathcal{J}_{n}\mathcal{X}$ of twisted $n$-jets is the disjoint union of the stacks 
$\mathcal{J}_{n}^{l}\mathcal{X}$ as $l$ varies over the positive integers. Both $\mathcal{J}_{n}\mathcal{X}$ and the $\mathcal{J}_{n}^{l}\mathcal{X}$ are smooth Deligne-Mumford stacks
(Theorem 2.9 \cite{YasMotivic}). The stack $\mathcal{J}_{0}\mathcal{X}$ is identified with the \emph{inertia stack} $\mathcal{I}(\mathcal{X})$ of $\mathcal{X}$. That is,  an object of 
$\mathcal{J}_{0}\mathcal{X}$ over $S$ is determined by a pair $(x, \alpha)$, where $x$ is an object of $\mathcal{X}$ over $S$ and $\alpha$ is an automorphism of $x$. 
We identify the stack $\mathcal{J}_{0}^{1}\mathcal{X}$  with $\mathcal{X}$. 
For any $m \geq n$ and for each $l > 0$, the truncation map $R[t]/(t^{ml + 1}) \rightarrow R[t]/(t^{nl + 1})$ 
induces a morphism $\mathcal{D}_{n,S}^{l} \rightarrow \mathcal{D}_{m,S}^{l}$.
Via composition, we get a natural affine morphism (Theorem 2.9 \cite{YasMotivic})
\begin{equation*}
\pi^{m}_{n} : \mathcal{J}_{m}\mathcal{X} \rightarrow \mathcal{J}_{n}\mathcal{X}. 
\end{equation*}
The projective system $\{  \mathcal{J}_{n} \mathcal{X} \}_{n}$ has a projective limit with a projection morphism (p15 \cite{YasMotivic})
\begin{equation*}
\mathcal{J}_{\infty}\mathcal{X} = \lim_{\leftarrow} \mathcal{J}_{n}\mathcal{X}.
\end{equation*}
\begin{equation*}
\pi_{n}: \mathcal{J}_{\infty}\mathcal{X} \rightarrow  \mathcal{J}_{n}\mathcal{X}.
\end{equation*}

\begin{rem}
The open covering $\{ \mathcal{X}(\bsigma) \, | \, \dim \sigma = d \}$ of $\mathcal{X}$ induces an open covering 
$\{ \mathcal{J}_{n}\mathcal{X}(\bsigma) \, | \, \dim \sigma = d \}$ of $\mathcal{J}_{n}\mathcal{X}$, for $0 \leq n \leq \infty$. 
\end{rem}

Recall that if $\mathcal{Y}$ is a stack over $\mathbb{C}$, then we can consider the set of \emph{points} $|\mathcal{Y}|$ of $\mathcal{Y}$ over $\mathbb{C}$ \cite{LMBChamps}. 
Its elements are equivalence 
classes of morphisms from $\Spec \mathbb{C}$ to $\mathcal{Y}$. Two morphisms 
$\psi, \psi': \Spec \mathbb{C} \rightarrow \mathcal{Y}$ are equivalent if there is a $2$-morphism from $\psi$ to $\psi'$. 
It has a Zariski topology;
for every open substack $\mathcal{Y'}$ of $\mathcal{Y}$, $|\mathcal{Y}'|$ is an open subset of $|\mathcal{Y}|$. If $\mathcal{Y}$ has a coarse moduli space $Y$, then 
$|\mathcal{Y}|$ is homeomorphic to $Y(\mathbb{C})$. We will often identify $Y$ with $Y(\mathbb{C})$.

Let $D_{\infty,\mathbb{C}} = \Spec \mathbb{C}[[t]]$ and $\mathcal{D}_{\infty,\mathbb{C}}^{l} = [D_{\infty,\mathbb{C}}/\mu_{l}]$. A \emph{twisted arc} of order $l$ of 
$\mathcal{X}$ over $\mathbb{C}$ 
is a representable morphism from $\mathcal{D}_{\infty,\mathbb{C}}^{l}$ to $\mathcal{X}$. Two twisted arcs $\alpha, \alpha': \mathcal{D}_{\infty,\mathbb{C}}^{l} \rightarrow \mathcal{X}$
of order $l$ are equivalent if there is a $2$-morphism from $\alpha$ to $\alpha'$. The set of equivalence classes of twisted arcs over $\mathcal{X}$ is 
identified with  $|\mathcal{J}_{\infty}\mathcal{X}|$ (p16 \cite{YasMotivic}), which we will call the 
\emph{space of twisted arcs} of $\mathcal{X}$. 

For $0 \leq n \leq \infty$ and a positive integer $l$, the scheme $D_{nl, \mathbb{C}}$  has a unique closed point.
A (not necessarily representable)  morphism 
$\gamma:  \mathcal{D}_{n,\mathbb{C}}^{l} \rightarrow \mathcal{X}$ induces an object  $\bar{\gamma}$
of $\mathcal{X}$ over $\mathbb{C}$,
\begin{equation*}
\bar{\gamma}: \Spec \mathbb{C} \rightarrow D_{nl, \mathbb{C}} \rightarrow \mathcal{D}_{n,\mathbb{C}}^{l} \rightarrow \mathcal{X}.
\end{equation*}
The automorphism group of the object in   $\mathcal{D}_{n,\mathbb{C}}^{l}$
corresponding to the morphism $\Spec \mathbb{C} \rightarrow  \mathcal{D}_{n,\mathbb{C}}^{l}$ is $\mu_{l}$. Hence we get a homomorphism $\phi: \mu_{l} \rightarrow 
\Aut(\bar{\gamma})$.
Since every automorphism in $\mathcal{D}_{n,\mathbb{C}}^{l}$ is locally induced by $\theta$, $\phi$ is injective if and only if $\Aut(\chi) \rightarrow \Aut(\gamma(\chi))$ is injective for all objects 
$\chi$ in $\mathcal{D}_{n,\mathbb{C}}^{l}$. This holds if and only if $\gamma$ is representable \cite{LMBChamps}. We conclude that $\gamma$ is representable if and only if $\phi$ is injective.

By considering coarse moduli spaces we have a commutative diagram
\begin{equation*}
\xymatrix{ \mathcal{D}_{n,\mathbb{C}}^{l}  \ar[d] \ar[r]^\gamma & \mathcal{X} \ar[d] \\
D_{n,\mathbb{C}} \ar[r]^{\gamma'} & X .}
\end{equation*}
We will consider the $n$th \emph{jet scheme} $J_{n}(X) = 
\Hom(D_{n,\mathbb{C}}, X)$ of $X$
and identify jet schemes with their $\mathbb{C}$-valued points. When $n = \infty$, $J_{\infty}(X)$ is called the \emph{arc space} of $X$. We have a map
\begin{equation}\label{smitten}
\tilde{\pi}_{n}: |\mathcal{J}_{n}\mathcal{X}| \rightarrow J_{n}(X)
\end{equation}
\begin{equation*}
\tilde{\pi}_{n}(\gamma) = \gamma'.
\end{equation*}

Fix a $d$-dimensional cone $\sigma$ of $\Sigma$ and consider the open substack $\mathcal{X}(\bsigma) = [\mathbb{A}^{d}/N(\sigma)]$ of $\mathcal{X}$. 
A twisted $n$-jet $\mathcal{D}_{n, S}^{l} \rightarrow \mathcal{X}(\bsigma)$
can be lifted to a morphism between atlases, $D_{nl,S} \rightarrow \mathbb{A}^{d}$. 
Yasuda used these lifts to describe the stack $\mathcal{J}_{n}^{l}\mathcal{X}$ (Proposition 2.8 \cite{YasMotivic}). We will present his result and a sketch of the proof in our situation.
We first fix some
notation.
The action of $\mu_{l}$ on
$D_{nl, \mathbb{C}}$ extends to an action on $J_{nl}(\mathbb{A}^{d})$. The action of $N(\sigma)$ on $\mathbb{A}^{d}$ also extends to an action on $J_{nl}(\mathbb{A}^{d})$. 
If $g$ is an element of $N(\sigma)$ of
order $l$, let $J_{nl}^{(g)}(\mathbb{A}^{d})$ be the closed subscheme of $J_{nl}(\mathbb{A}^{d})$ on which the actions of $\zeta_{l}$ in $\mu_{l}$ and $g$ in $N(\sigma)$ agree.

\begin{lemma}[\cite{YasMotivic}]\label{lemon}
For $0 \leq n \leq \infty$, we have a homeomorphism
\begin{equation*}
|\mathcal{J}_{n}\mathcal{X}(\bsigma)| \cong \coprod_{g \in N(\sigma)} J_{nl}^{(g)}(\mathbb{A}^{d})/N(\sigma).
\end{equation*}
\end{lemma}
\begin{proof}
We will only show that there is a bijection between the two sets.
We have noted that a representable morphism $\gamma: \mathcal{D}_{n,\mathbb{C}}^{l} \rightarrow [\mathbb{A}^{d}/N(\sigma)]$ is determined by the images of $\alpha$ and $\theta$. These 
images have the form
\begin{equation*}
\xymatrix{  &  & \mathbb{A}^{d} \\
D_{nl,\mathbb{C}} \times N(\sigma) \ar[urr]^{(v,\lambda) \mapsto \nu(v) \cdot \lambda} \ar[d]^{pr_{1}} \ar[r]_{\zeta_{l} \times g^{-1}} & D_{nl,\mathbb{C}} \times N(\sigma) 
\ar[ur]_{(v,\lambda) \mapsto \nu(v) \cdot \lambda} \ar[d]^{pr_{1}} & \\
D_{nl,\mathbb{C}} \ar[r]^{\zeta_{l}} & D_{nl,\mathbb{C}}   & }
\end{equation*}
for some $g$ in $N(\sigma)$ of order $l$ and some $nl$-jet $\nu: D_{nl,\mathbb{C}} \rightarrow \mathbb{A}^{d}$. 
Conversely, given such a diagram, we can construct a representable morphism. The diagram is determined by any choice of $g$ and $\nu$ satisfying
\begin{equation*}
\xymatrix{ D_{nl,\mathbb{C}} \ar[d]^{\zeta_{l}} \ar[r]^{\nu} & \mathbb{A}^{d} \ar[d]^{g} \\
D_{nl,\mathbb{C}} \ar[r]^{\nu} & \mathbb{A}^{d}.
}  
\end{equation*}
That is, $\gamma$ is determined by the pair $(g, \nu)$,  where $g$ has order $l$ and $\nu$ lies in $J_{nl}^{(g)}(\mathbb{A}^{d})$. 
Suppose $\gamma'$ is a twisted $n$-jet of order $l$ determined by the pair $(g', \nu')$. A $2$-morphism from $\gamma$ to $\gamma'$ is determined by a morphism
$\beta: \gamma(\alpha) \rightarrow \gamma'(\alpha)$ in $[\mathbb{A}^{d}/N(\sigma)]$ over the identity morphism on $D_{nl,\mathbb{C}}$,  such that  the following diagram commutes
\[
\xymatrix{ \gamma(\alpha) \ar[d]^{\gamma(\theta)} \ar[r]^{\beta} & \gamma'(\alpha) \ar[d]^{\gamma'(\theta)} \\
\gamma(\alpha) \ar[r]^{\beta} & \gamma'(\alpha).
}  
\]
One verifies that the morphism $\beta$ is determined by an element $h$ in $N(\sigma)$ such that $\nu = \nu' \cdot h$ and that the diagram above is commutative if and only if $g = g'$. 
Hence the equivalence class of $\gamma$ in $|\mathcal{J}_{n}\mathcal{X}(\bsigma)|$ corresponds to a point in $J_{nl}^{(g)}(\mathbb{A}^{d})/N(\sigma)$. This gives our desired bijection. 

\end{proof}

\begin{rem}\label{doko}
Borisov, Chen and Smith gave a decomposition of $|\mathcal{I}\mathcal{X}|$ into connected components indexed by $\BOX(\bSigma)$ (Proposition 4.7 \cite{BCSOrbifold}). 
Given $v$ in $\BOX(\bSigma)$, let $\sigma(v)$ be the cone containing $v$ in its relative interior and 
let $\Sigma_{\sigma(v)}$ be the simplicial fan in $(N/N_{\sigma(v)})_{\mathbb{R}}$ with cones given by the projections 
of the cones in $\Sigma$ containing $\sigma(v)$. The associated toric variety $X(\Sigma_{\sigma(v)})$ is a $T$-invariant closed subvariety of $X$. 
The connected component 
of $|\mathcal{I}\mathcal{X}|$ corresponding
to $v$ is homeomorphic to the simplicial toric variety $X(\Sigma_{\sigma(v)})$. By taking $n = 0$ in Lemma \ref{lemon}, we recover this decomposition for 
$|\mathcal{I}\mathcal{X}(\bsigma)|$.
\end{rem}

Consider an element $g$ in $N(\sigma)$ of order $l$ 
corresponding to $v$ in $\BOX(\btau)$, for some cone $\tau$ contained in $\sigma$. We will give an explicit description of $J_{nl}^{(g)}(\mathbb{A}^{d})$.
If we write
$v = \sum_{i = 1}^{d} q_{i} b_{i}$,
for some $0 \leq q_{i} < 1$, then recall from (\ref{actions}) that the action of $g$ on $\mathbb{A}^{d}$ is given by 
\begin{equation*}
g \cdot (x_{1}, \ldots, x_{d}) = (\zeta_{l}^{\,lq_{1}}x_{1}, \ldots , \zeta_{l}^{\,lq_{d}}x_{d}). 
\end{equation*}
An element $\nu$ of $J_{nl}(\mathbb{A}^{d})$ can be written in the form
\begin{equation*}
\nu = ( \, \sum_{j= 0}^{nl} \alpha_{1,j}t^{j}\, , \ldots , \, \sum_{j= 0}^{nl} \alpha_{d,j}t^{j} \, ),
\end{equation*}
for some $\alpha_{i,j}$ in $\mathbb{C}$, $1 \leq i \leq d$, $0 \leq j \leq nl$. We have
\begin{equation*}
g \cdot \nu =  ( \, \zeta_{l}^{\,lq_{1}} \sum_{j= 0}^{nl} \alpha_{1,j}t^{j}\, , \ldots , \, \zeta_{l}^{\,lq_{d}} \sum_{j= 0}^{nl} \alpha_{d,j}t^{j} \, )
\end{equation*}
\begin{equation*}
\zeta_{l} \cdot \nu = ( \, \sum_{j= 0}^{nl} \alpha_{1,j}\zeta_{l}^{j}t^{j}\, , \ldots , \, \sum_{j= 0}^{nl}\alpha_{d,j} \zeta_{l}^{j}t^{j} \, ) .
\end{equation*}
Hence $v$ lies in $J_{nl}^{(g)}(\mathbb{A}^{d})$ if and only if $\alpha_{i,j} = 0$ whenever $j \not\equiv lq_{i} \, \left(\textrm{mod } l \right)$. We conclude that
\begin{equation}\label{weagles}
J_{nl}^{(g)}(\mathbb{A}^{d}) = \{     ( \, \sum_{j= 0}^{nl} \alpha_{1,j}t^{j}\, , \ldots , \, \sum_{j= 0}^{nl} \alpha_{d,j}t^{j} \, )      \mid
\alpha_{i,j} = 0 \textrm{ if } j \not\equiv lq_{i} \, \left(\textrm{mod } l \right) \}.
\end{equation}

\section{Twisted Arcs of Toric Stacks}\label{arc}

In this section we describe an action of $J_{\infty}(T)$ on an open, dense subset of $|\mathcal{J}_{\infty}\mathcal{X}|$ and compute the corresponding orbits and orbit closures.
 
We first recall Ishii's decomposition of the arc space of the toric variety $X$ \cite{IshArc}. Recall that $D_{1}, \ldots, D_{r}$ denote the $T$-invariant prime divisors of $X$ and
let $J_{\infty}(X)' = J_{\infty}(X) \smallsetminus \cup_{i} J_{\infty}(D_{i})$. The action of $T$ on $X$ extends to an action of $J_{\infty}(T)$ on $J_{\infty}(X)'$
(Proposition 2.6 \cite{IshArc}). Given an arc $\gamma$ in $J_{\infty}(X)'$, we can find a $d$-dimensional cone $\sigma$ such that
$\gamma: \Spec \mathbb{C}[[t]] \rightarrow U_{\sigma} \subseteq X$
corresponds to the ring homomorphism
$\gamma^{\#}: \mathbb{C}[\sigma^{\vee} \cap M] \rightarrow \mathbb{C}[[t]]$.
We have an induced semigroup morphism
$\sigma^{\vee} \cap M \rightarrow \mathbb{N}$,
$u \mapsto \ord_{t} \gamma^{\#}(\chi^{u})$,
which extends to  homomorphism from $M$ to $\mathbb{Z}$, necessarily of the form $\langle \: , v \rangle$, for some 
$v$ in $\sigma \cap N$. 
 Consider the arc
\begin{equation*}
\gamma_{v} : \Spec \mathbb{C}[[t]] \rightarrow U_{\sigma} \subseteq X
\end{equation*}
corresponding to the ring homomorphism
\begin{equation*}
\gamma_{v}^{\#}: \mathbb{C}[\sigma^{\vee} \cap M] \rightarrow \mathbb{C}[[t]]
\end{equation*}
\begin{equation*}
\chi^{u}  \mapsto t^{\langle u , v \rangle}.  
\end{equation*}
If $\phi$ denotes the arc in $J_{\infty}(T)$ corresponding to the ring homomorphism
$\phi^{\#}: \mathbb{C}[\sigma^{\vee} \cap M] \rightarrow \mathbb{C}[[t]]$,
$\chi^{u}  \mapsto \gamma^{\#}(\chi^{u})/t^{\langle u , v \rangle}$,
then $\gamma = \gamma_{v} \cdot \phi$ and both $v$ in $|\Sigma| \cap N$ and $\phi$ in $J_{\infty}(T)$ are uniquely determined. 
We have shown the first part of the following theorem.
\begin{theorem}[\cite{IshArc}]\label{Ishii}
We have a decomposition of $J_{\infty}(X)'$ into $J_{\infty}(T)$-orbits 
\begin{equation*}
J_{\infty}(X)' = \coprod_{v \in |\Sigma| \cap N} \gamma_{v} \cdot J_{\infty}(T).
\end{equation*}
Moreover, $\gamma_{w} \cdot J_{\infty}(T) \subseteq \overline{\gamma_{v} \cdot J_{\infty}(T)}$ if and only if $w - v$ lies in 
some cone $\sigma$ containing $v$ and $w$.
\end{theorem}

We will give a similar decomposition of the space $|\mathcal{J}_{\infty}\mathcal{X}|$ of twisted arcs of $\mathcal{X}$. For $i = 1, \ldots, r$, we first describe a closed substack 
$\mathcal{D}_{i}$ of $\mathcal{X}$ with coarse moduli space $D_{i}$ \cite{BCSOrbifold}. Fix a maximal cone $\sigma$. 
If $\rho_{i}$ is not in $\sigma$ then 
$\mathcal{D}_{i}$ is disjoint from $\mathcal{X}(\bsigma)$. Suppose $\rho_{i}$ is a ray of $\sigma$ and 
consider the projection $p_{i}: N \rightarrow N(\rho_{i}) = N / N_{\rho_{i}}$,
where $N_{\rho_{i}}$ is the sublattice of $N$ generated by $b_{i}$. 
If $\sigma_{i}$ denotes the cone $p_{i}(\sigma)$ in $N(\rho_{i})$,
then $(N(\rho_{i}), \sigma_{i}, \{ p_{i}(b_{j}) \}_{\rho_{j} \subseteq \sigma} )$
is the stacky fan corresponding to 
$\mathcal{D}_{i} \cap \mathcal{X}(\bsigma)$\footnote{Note that 
$N(\rho_{i})$ may have torsion and so $\sigma_{i}$ is a cone in the image of $N(\rho_{i})$ in $N(\rho_{i})_{\mathbb{Q}}$. There are no difficulties in generalising to this situation.
See Section \ref{remedy} for a discussion of this issue. This is the level of generality used in \cite{BCSOrbifold}.}. 
Note that $p_{i}$ induces an isomorphism between $N(\sigma) = N/N_{\sigma}$ and 
$N(\rho_{i})(\sigma_{i}) =  N(\rho_{i}) / N(\rho_{i})_{\sigma_{i}}$. We have an inclusion of $\mathbb{A}^{d - 1}$ into $\mathbb{A}^{d}$
by setting the coordinate corresponding to 
$\rho_{i}$ to be zero.
We conclude that  
$\mathcal{D}_{i} \cap \mathcal{X}(\bsigma) \cong [\mathbb{A}^{d - 1}/ N(\sigma)]$ and the inclusion of $\mathbb{A}^{d - 1}$ into $\mathbb{A}^{d}$ 
induces the inclusion of  $\mathcal{D}_{i} \cap \mathcal{X}(\bsigma)$ into $\mathcal{X}(\bsigma)$. Moreover, by Lemma \ref{lemon}, 
if $g$ is an element of $N(\sigma)$, then
we have an induced inclusion
\begin{equation}\label{very}
J_{\infty}^{(g)}(\mathbb{A}^{d - 1})/N(\sigma) \hookrightarrow J_{\infty}^{(g)}(\mathbb{A}^{d})/N(\sigma),
\end{equation}
corresponding to the closed inclusion $|\mathcal{J}_{\infty}(\mathcal{D}_{i} \cap \mathcal{X}(\bsigma))| \hookrightarrow 
|\mathcal{J}_{\infty}\mathcal{X}(\bsigma)|$. We define $|\mathcal{J}_{\infty}\mathcal{X}|'$ to be the open subset
\begin{equation*}
|\mathcal{J}_{\infty}\mathcal{X}|' = |\mathcal{J}_{\infty}\mathcal{X}| \smallsetminus \cup_{i = 1}^{r} |\mathcal{J}_{\infty}\mathcal{D}_{i}|.
\end{equation*}
Similarly, we consider $|\mathcal{J}_{\infty}\mathcal{X}(\bsigma)|'$ and let $J_{\infty}(\mathbb{A}^{d})'$ be the open locus of arcs of  $\mathbb{A}^{d}$ that are not contained in a coordinate hyperplane. Setting $J_{\infty}^{(g)}(\mathbb{A}^{d})' = J_{\infty}^{(g)}(\mathbb{A}^{d}) \cap J_{\infty}(\mathbb{A}^{d})'$,
it follows from Lemma \ref{lemon} and
 (\ref{very}) that
\begin{equation}\label{dash}
|\mathcal{J}_{\infty}\mathcal{X}(\bsigma)|' \cong \coprod_{g \in N(\sigma)} J_{\infty}^{(g)}(\mathbb{A}^{d})'/N(\sigma).
\end{equation}
We will often make this identification.

For a fixed positive integer $l$, 
the open embedding of $T$ in $\mathcal{X}(\bsigma)$ extends to an open embedding of $J_{\infty}(T) \cong \mathcal{J}_{\infty}^{l}T$ in 
$\mathcal{J}_{\infty}^{l}\mathcal{X}(\bsigma)$. We obtain an action of $J_{\infty}(T)$ on $|\mathcal{J}_{\infty}^{l}\mathcal{X}(\bsigma)|$, which 
restricts to an action on $|\mathcal{J}_{\infty}^{l}\mathcal{X}(\bsigma)|'$. More specifically, we identify
$J_{\infty}(T) \cong J_{\infty}^{(l)}(T)/N(\sigma)$,
where $J_{\infty}^{(l)}(T)$ is the closed subscheme of $J_{\infty}(T)$ fixed by the action of $\mu_{l}$. In coordinates, 
\begin{equation*}
J_{\infty}^{(l)}(T) =  ( \, \sum_{j= 0}^{\infty} \beta_{1,j}t^{lj}\, , \ldots , \, \sum_{j= 0}^{\infty} \beta_{d,j}t^{lj} \, )      \mid \beta_{i,0} \neq 0 \textrm{ for }
i = 1, \ldots, d \}.
\end{equation*}
Fix an element $g$ in $N(\sigma)$ of order $l$ and recall from (\ref{weagles}) that 
\begin{equation*} 
J_{\infty}^{(g)}(\mathbb{A}^{d}) = \{     ( \, \sum_{j= 0}^{\infty} \alpha_{1,j}t^{j}\, , \ldots , \, \sum_{j= 0}^{\infty} \alpha_{d,j}t^{j} \, )      \mid
\alpha_{i,j} = 0 \textrm{ if } j \not\equiv lq_{i} \, \left( \textrm{mod } l \right) \}.
\end{equation*}
The elements in $J_{\infty}^{(g)}(\mathbb{A}^{d})'$ also satisfy the property that for each $i = 1, \ldots ,  d$, there exists a non-negative integer $j$ such that
 $\alpha_{i,j} \neq 0$.  
Componentwise multiplication of power series gives an
action, 
\begin{equation*}
J_{\infty}^{(l)}(T)/N(\sigma) \times J_{\infty}^{(g)}(\mathbb{A}^{d})'/N(\sigma) \rightarrow J_{\infty}^{(g)}(\mathbb{A}^{d})'/N(\sigma).
\end{equation*}
By (\ref{dash}), this induces an action of $J_{\infty}(T)$ on $|\mathcal{J}_{\infty}^{l}\mathcal{X}(\bsigma)|'$. By varying $l$, we obtain an 
action of 
$J_{\infty}(T)$ on $|\mathcal{J}_{\infty}\mathcal{X}|'$. We would like to describe the $J_{\infty}(T)$-orbits. 

Let $w$ be an element of $\sigma \cap N$. There is a unique decomposition 
$w = v + \sum_{i = 1}^{d} \lambda_{i}b_{i}$,
where $v$ lies in $\BOX(\btau)$, for some $\tau \subseteq \sigma$, and the $\lambda_{i}$ are non-negative integers. We will use the notation $\{ w \} = v$.
Suppose $v$ corresponds to an element $g$ in $N(\sigma)$ of order $l$. 
If we write
$v = \sum_{i = 1}^{d} q_{i} b_{i}$, for some $0 \leq q_{i} < 1$, then
$w = \sum_{i = 1}^{d} w_{i} b_{i} = \sum_{i = 1}^{d} (\lambda_{i} + q_{i}) b_{i}$, 
where $w_{i} = \lambda_{i} + q_{i}$. 
We define
$\tilde{\gamma}_{w} \in |\mathcal{J}_{\infty}^{l}\mathcal{X}(\bsigma)|' \subseteq |\mathcal{J}_{\infty}\mathcal{X}|'$
to be the equivalence class of twisted arcs corresponding to the element 
\begin{equation}\label{MickyO}
( t^{ l(\lambda_{1} + q_{1}) }, \ldots , t^{ l(\lambda_{d} + q_{d}) } ) = ( t^{ lw_{1}}, \ldots ,  t^{ lw_{d}} ) 
\end{equation}
of $J_{\infty}^{(g)}(\mathbb{A}^{d})'$, under the isomorphism (\ref{dash}).

\begin{theorem}\label{decomp}
We have a decomposition of $|\mathcal{J}_{\infty}\mathcal{X}|'$ into $J_{\infty}(T)$-orbits 
\begin{equation*}
|\mathcal{J}_{\infty}\mathcal{X}|' = \coprod_{v \in |\Sigma| \cap N} \tilde{\gamma}_{v} \cdot J_{\infty}(T).
\end{equation*}
Moreover, $\tilde{\gamma}_{w} \cdot J_{\infty}(T) \subseteq \overline{\tilde{\gamma}_{v} \cdot J_{\infty}(T)}$ if and only if 
$w - v = \sum_{\rho_{i} \subseteq \sigma} \lambda_{i}b_{i}$ for some non-negative integers $\lambda_{i}$ and  
some cone $\sigma$ containing $v$ and $w$. With the notation of Theorem \ref{Ishii} and (\ref{smitten}),
\begin{equation*}
\tilde{\pi}_{\infty}: |\mathcal{J}_{\infty}\mathcal{X}|' \rightarrow J_{\infty}(X)'
\end{equation*}
is a $J_{\infty}(T)$-equivariant bijection satisfying
\begin{equation*}
\tilde{\pi}_{\infty}(\tilde{\gamma}_{v}) = \gamma_{v}.
\end{equation*}
\end{theorem}
\begin{proof}
Let $\sigma$ be a $d$-dimensional cone in $\Sigma$ and let $g$ be an element in $N(\sigma)$ of order $l$ 
corresponding to an element $v$ in $\BOX(\btau)$, for some $\tau \subseteq \sigma$. 
By (\ref{weagles}), and with the notation of the previous discussion, we have a decomposition
\begin{equation*}
J_{\infty}^{(g)}(\mathbb{A}^{d})' = \coprod_{  \substack{w \in \sigma \cap N \\ \{w\} = v}}  ( t^{ lw_{1}}, \ldots ,  t^{ lw_{d}} ) \cdot J_{\infty}^{(l)}(T),
\end{equation*}
and hence 
\begin{equation*}
J_{\infty}^{(g)}(\mathbb{A}^{d})'/N(\sigma) \cong \coprod_{  \substack{w \in \sigma \cap N \\ \{w\} = v}}  \tilde{\gamma}_{w} \cdot J_{\infty}(T).
\end{equation*}
Also, 
$\tilde{\gamma}_{w} \cdot J_{\infty}(T) \subseteq \overline{\tilde{\gamma}_{w'} \cdot J_{\infty}(T)}$ in 
$J_{\infty}^{(g)}(\mathbb{A}^{d})'/N(\sigma)$ if and only if $w_{i} \geq w_{i}'$ for $i = 1, \ldots, d$.
We conclude that 
\begin{equation*}
|\mathcal{J}_{\infty}\mathcal{X}(\bsigma)|' = \coprod_{w \in \sigma \cap N} \tilde{\gamma}_{w} \cdot J_{\infty}(T),
\end{equation*}
and $\tilde{\gamma}_{w} \cdot J_{\infty}(T) \subseteq \overline{\tilde{\gamma}_{w'} \cdot J_{\infty}(T)}$ in 
$|\mathcal{J}_{\infty}\mathcal{X}(\bsigma)|'$ if and only if $\{w\} = \{w'\}$ and $w_{i} \geq w_{i}'$ for $i = 1, \ldots, d$. This is equivalent to 
$w - w' = \sum_{i = 1}^{d} \lambda_{i}b_{i}$ for some non-negative integers $\lambda_{i}$. Since the open subsets $|\mathcal{J}_{\infty}\mathcal{X}(\bsigma)|'$ cover
$|\mathcal{J}_{\infty}\mathcal{X}|'$ as $\sigma$ varies over all maximal cones, we obtain the desired decomposition and closure relations.

Consider the sublattice $N_{\sigma} \subseteq N$ and recall that $\sigma'$ is the cone in $N_{\sigma}$ generated by $b_{1}, \ldots, b_{d}$.
If $H$ denotes the quotient of $N$ by the sublattice generated by $v_{1}, \ldots, v_{d}$, then
we have a pairing
$\langle \; , \, \rangle : N \times M_{\sigma} \rightarrow \mathbb{Q}$,
$\langle  v_{i}, u_{j} \rangle = \delta_{i,j}|H|$,
where $\delta_{i,j} = 1$ if $i = j$ and $0$ otherwise. 
If $u_{1}, \ldots, u_{d}$ are the primitive integer generators of
$\sigma^{\vee}$ in $M$, then $u_{1}/a_{1}|H|, \ldots, u_{d}/a_{d}|H|$ are the primitive integer generators of $(\sigma')^{\vee}$ in $M_{\sigma}$.
We have made an identification throughout that
$\Spec \mathbb{C}[(\sigma')^{\vee} \cap M_{\sigma}] \cong \mathbb{A}^{d}$.
Consider an element $\gamma$ in $J_{\infty}^{(g)}(\mathbb{A}^{d})'$, 
for some $g$ in $N(\sigma)$ of order $l$, corresponding to an element of
$|\mathcal{J}_{\infty}\mathcal{X}(\bsigma)|'$.
Applying $\tilde{\pi}_{\infty}$ 
gives an arc in $U_{\sigma}$, which we denote by $\tilde{\pi}_{\infty}(\gamma)$. We have a commutative diagram
\begin{equation*}
\xymatrix{ \mathbb{C}[(\sigma')^{\vee} \cap M_{\sigma}]  \ar[r]^{\: \; \: \gamma^{\#}} & \mathbb{C}[[t]]  \\
\mathbb{C}[\sigma^{\vee} \cap M] \ar[r]^{\: \: \: \: \tilde{\pi}_{\infty}(\gamma)^{\#}} \ar[u] & \mathbb{C}[[t]] \ar[u]_{t \mapsto t^{l}}.
}  
\end{equation*}
It follows that $\tilde{\pi}_{\infty}$ is $J_{\infty}(T)$-equivariant. Let $w$ be an element of $\sigma \cap N$ and consider
the notation of (\ref{MickyO}). With a slight abuse of notation,
$(\tilde{\gamma}_{w})^{\#}(\chi^{u_{i}/a_{i}|H|})  = t^{lw_{i}}$,
and we compute 
\begin{equation*}
\tilde{\pi}_{\infty}(\tilde{\gamma}_{w})^{\#}(\chi^{u_{i}}) = ((t^{1/l})^{lw_{i}})^{a_{i}|H|} = t^{a_{i}w_{i}|H|} = t^{\langle u_{i}, w \rangle}. 
\end{equation*}
It follows from the definition of $\gamma_{w}$ that $\tilde{\pi}_{\infty}(\tilde{\gamma}_{w}) = \gamma_{w}$. 
\end{proof}

\begin{rem}
The morphism from a Deligne-Mumford stack to its coarse moduli space is proper \cite{LMBChamps}. Hence the fact that the map
$\tilde{\pi}_{\infty}: |\mathcal{J}_{\infty}\mathcal{X}|' \rightarrow J_{\infty}(X)'$
is bijective follows from Proposition 3.37 of \cite{YasMotivic}. 
\end{rem}

\begin{rem}\label{ginvo}
We give a geometric description of a well-known involution $\iota$ on $|\Sigma| \cap N$ (see, for example, Section 2 \cite{YoWeightI}). 
If $\tau$ is a cone in $\Sigma$ and $v$ is a lattice point in $\BOX(\btau)$, then $v$ can be uniquely written in the form 
$v = \sum_{\rho_{i} \subseteq \tau} q_{i}b_{i}$, for some $0< q_{i} <1$. We define 
$\iota = \iota_{\bSigma} : \BOX(\btau) \rightarrow \BOX(\btau)$  by 
$\iota(v) = \sum_{\rho_{i} \subseteq \tau} (1 - q_{i})b_{i}$.
By Remark \ref{doko}, the connected components of $|\mathcal{I}\mathcal{X}|$ are indexed by $\BOX(\bSigma)$. Recall that the elements of 
$|\mathcal{I}\mathcal{X}|$ are equivalence classes of pairs $(x, \alpha)$, where $x$ is an object of $\mathcal{X}$ over $\mathbb{C}$ and $\alpha$ is an automorphism of $x$.
There is a natural involution on  $|\mathcal{I}\mathcal{X}|$, taking a pair $(x, \alpha)$ to $(x, \alpha^{-1})$.
Observe that $I$ induces the involution $\iota$ on the connected components of $|\mathcal{I}\mathcal{X}|$. 

Every lattice point $w$ in $|\Sigma| \cap N$ can be uniquely written in the form 
$w = \{ w \} + \tilde{w}$, where $\{ w \}$ lies in $\BOX(\btau)$ for some $\tau \subseteq \sigma(w)$ and $\tilde{w}$ lies in $N_{\sigma(w)}$. Here 
$\sigma(w)$ is the cone of $\Sigma$ containing $w$ in
its relative interior. The map
$\iota$ extends to an involution on $|\Sigma| \cap N$, taking $w$ to $\iota(\{ w \}) + \tilde{w}$.
Recall that we have a projection morphism
$\pi: \mathcal{J}_{\infty}\mathcal{X} \rightarrow  \mathcal{I}\mathcal{X} = \mathcal{J}_{0}\mathcal{X}$.
We see that $I$ extends to a $J_{\infty}(T)$-equivariant involution 
$I : |\mathcal{J}_{\infty}\mathcal{X}|' \rightarrow |\mathcal{J}_{\infty}\mathcal{X}|'$ satisfying
$I(\tilde{\gamma}_{v}) = \tilde{\gamma}_{\iota(v)}$.
\end{rem}

\begin{rem}
For any non-negative integer $n$, we can consider 
$|\mathcal{J}_{n}\mathcal{X}|' = |\mathcal{J}_{n}\mathcal{X}| \smallsetminus \cup_{i = 1}^{r} |\mathcal{J}_{n}\mathcal{D}_{i}|$
and an action of $J_{n}(T)$ on $|\mathcal{J}_{n}\mathcal{X}|'$. Recall that we have projection morphisms 
$\pi_{n}: \mathcal{J}_{\infty}\mathcal{X} \rightarrow  \mathcal{J}_{n}\mathcal{X}$.
For each non-zero cone $\tau$ of $\Sigma$, let
\begin{equation*}
\overline{\BOX}(\bntau) = \{ v \in N \mid  v = \sum_{\rho_{i} \subseteq \tau} q_{i}b_{i} \textrm{  for some  }
0 < q_{i} \leq n \}.
\end{equation*}
We set $\overline{\BOX}(\mbox{\boldmath$n\{ 0 \}$}) = \{0\}$ and 
$\overline{\BOX}(\bnSigma) = \cup_{\tau \in \Sigma} \overline{\BOX}(\bntau)$. It can be shown that there is a decomposition of $|\mathcal{J}_{n}\mathcal{X}|'$ into 
$J_{n}(T)$-orbits
\begin{equation*}
|\mathcal{J}_{n}\mathcal{X}|' = \coprod_{v \in \overline{\BOX}(\bnSigma)} \pi_{n}(\tilde{\gamma}_{v}) \cdot J_{n}(T).
\end{equation*}
\end{rem}

\section{Contact Order along a Divisor}\label{contact}

In this section, we compute the contact order of a twisted arc along a $T$-invariant divisor on $\mathcal{X}$. 
Recall that for each maximal cone $\sigma$ in $\Sigma$, we have morphisms
\begin{equation*}
\mathbb{A}^{d} \stackrel{p}{\rightarrow}   [\mathbb{A}^{d}/N(\sigma)] \cong   \mathcal{X}(\bsigma)  \stackrel{q}{\rightarrow} 
  \mathbb{A}^{d}/N(\sigma) \cong   U_{\sigma} ,
\end{equation*}
where $\mathbb{A}^{d}$ is the atlas of $\mathcal{X}(\bsigma)$ via $p$ and $U_{\sigma}$ is the coarse moduli space of $\mathcal{X}(\bsigma)$.
For every ray $\rho_{i}$ in $\Sigma$, there is a corresponding divisor $\mathcal{D}_{i}$ of $\mathcal{X}(\bSigma)$ (see Section \ref{arc}). If $\rho_{i} \nsubseteq \sigma$ then 
$\mathcal{D}_{i}$ does not intersect 
$\mathcal{X}(\bsigma)$. If $\rho_{i} \subseteq \sigma$, then $\mathcal{D}_{i}$ corresponds to the divisor
$\{ x_{i} = 0 \}$ on $\mathbb{A}^{d}$, with an appropriate choice of coordinates. 
We say that a $\mathbb{Q}$-divisor $\mathcal{E}$ on $\mathcal{X}$ is \emph{$T$-invariant}  if $\mathcal{E} = \sum_{i} \beta_{i} \mathcal{D}_{i}$, for some $\beta_{i} \in \mathbb{Q}$.
There is a natural isomorphism of Picard groups over $\mathbb{Q}$ (Example 6.7 \cite{VisIntersection})
\begin{equation*}
q^{*}: \Pic_{\mathbb{Q}} (X(\Sigma) ) \rightarrow     \Pic_{\mathbb{Q}} (\mathcal{X}(\bSigma) ) 
\end{equation*}
\begin{equation*}
 D_{i} \mapsto   q^{*}D_{i} = a_{i}\mathcal{D}_{i}  .
\end{equation*}
The inverse map is induced by the pushforward functor
$q_{*}: \Pic_{\mathbb{Q}} (\mathcal{X}(\bSigma) ) \rightarrow \Pic_{\mathbb{Q}} (X(\Sigma) )$. 
The canonical divisor on $\mathcal{X}$ is given by
$K_{\mathcal{X}} = - \sum_{i} \mathcal{D}_{i}$,
and so
$q_{*}K_{\mathcal{X}} = - \sum_{i} a_{i}^{-1}D_{i}$. 
Recall that a $T$-invariant $\mathbb{Q}$-divisor $E = \sum_{\rho_{i} \in \Sigma} \alpha_{i} D_{i}$ on $X$ corresponds to a
real-valued piecewise $\mathbb{Q}$-linear function 
$\psi_{E}: |\Sigma| \rightarrow \mathbb{R}$ satisfying $\psi_{E}(v_{i}) = -\alpha_{i}$ \cite{FulIntroduction}. 
Note that 
$\psi_{q_{*}K_{\mathcal{X}}}(b_{i})  = 1$  for $i = 1, \ldots,r$, 
and hence, with the notation of Section \ref{tstack}, $\psi = \psi_{q_{*}K_{\mathcal{X}}}$.
Let $\mathcal{E} = \sum_{i} \beta_{i} \mathcal{D}_{i}$ be a $T$-invariant $\mathbb{Q}$-divisor on $\mathcal{X}$. Following Yasuda (Definition 4.17 \cite{YasMotivic}), we say that the pair 
$(\mathcal{X}, \mathcal{E})$ is \emph{Kawamata log terminal} if $\beta_{i} < 1$ for $i = 1, \ldots, r$. Geometrically, this says that for each maximal cone $\sigma$, the representative of $\mathcal{E}$ in the atlas $\mathbb{A}^{d}$ of $\mathcal{X}(\bsigma) = [\mathbb{A}^{d}/N(\sigma)]$
is supported on the coordinate axes with all coefficients less than $1$. Equivalently, the condition says that
$\psi_{q_{*}\mathcal{E}}(b_{i}) > -1$ for $i = 1, \ldots,r$.

If $\mathcal{E} = \sum u_{j} \mathcal{E}_{j}$  is a $\mathbb{Q}$-divisor on $\mathcal{X}$, for some prime divisors $\mathcal{E}_{j}$ and rational numbers $u_{j}$, then 
Yasuda \cite{YasMotivic} defined a function
\begin{equation*}
\ord \mathcal{E}: |\mathcal{J}_{\infty} \mathcal{X}| \smallsetminus  \cup_{j} |\mathcal{J}_{\infty} \mathcal{E}_{j}| \rightarrow \mathbb{Q},
\end{equation*}
\begin{equation*}
\ord \mathcal{E} = \sum_{j} u_{j} \ord \mathcal{E}_{j}.
\end{equation*}
When $\mathcal{E}$ is a prime divisor, the function  $\ord \mathcal{E}$ is defined as follows:
if $\gamma: \mathcal{D}_{\infty,\mathbb{C}}^{l} \rightarrow \mathcal{X}$ is a representable morphism, 
then consider the composition of $\gamma$ with the atlas $\Spec \mathbb{C}[[t]] \rightarrow \mathcal{D}_{\infty,\mathbb{C}}^{l}$, 
and choose a lifting to an arc $\bar{\gamma}$ 
of an atlas $M$ of $\mathcal{X}$. If $m$ is the contact order of $\bar{\gamma}$ along the representative 
of $\mathcal{E}$ in $M$, then
$\ord \mathcal{E}(\gamma) = m/l$.

The following lemma describes the function $\ord \mathcal{E}$ in the toric case.
We will use the notation of Theorem \ref{decomp}.

\begin{lemma}\label{order}
If $\mathcal{E}$ is a $T$-invariant $\mathbb{Q}$-divisor on $\mathcal{X}$,
then
\begin{equation*}
\ord \mathcal{E}( \tilde{\gamma}_{w} \cdot J_{\infty}(T) ) = -\psi_{q_{*}\mathcal{E}}(w).
\end{equation*}
In particular, 
\begin{equation*}
\ord K_{\mathcal{X}}( \tilde{\gamma}_{w} \cdot J_{\infty}(T) ) = -\psi(w).
\end{equation*}
\end{lemma}
\begin{proof}
It will be enough to prove the result for $\mathcal{E} = \mathcal{D}_{i}$ and twisted arcs of the form $\tilde{\gamma}_{w}$. 
Consider the representable morphism
$\tilde{\gamma}_{w}: \mathcal{D}_{\infty,\mathbb{C}}^{l} \rightarrow \mathcal{X}(\bsigma) \subseteq \mathcal{X}$,
for some lattice point $w$ in a maximal cone $\sigma$. 
With the notation of (\ref{MickyO}), the composition of $\tilde{\gamma}_{w}$ with $\Spec \mathbb{C}[[t]] \rightarrow \mathcal{D}_{\infty,\mathbb{C}}^{l}$ lifts to an arc
$( t^{ lw_{1}}, \ldots ,  t^{ lw_{d}} )$ 
in $J_{\infty}(\mathbb{A}^{d})$. If $\rho_{i} \nsubseteq \sigma$ then $\ord \mathcal{D}_{i} (\tilde{\gamma}_{w}) = 0 = -\psi_{a_{i}^{-1}D_{i}}(w)$. If $\rho_{i} \subseteq \sigma$
then the divisor $\mathcal{D}_{i}$ is represented by the divisor $\{ x_{i} = 0 \}$ in the atlas $\mathbb{A}^{d}$ and we conclude that
$\ord \mathcal{D}_{i}( \tilde{\gamma}_{w}) = w_{i} = -\psi_{a_{i}^{-1}D_{i}}(w)$.
The second statement follows since $\psi = \psi_{q_{*}K_{\mathcal{X}}}$.
\end{proof}

\section{Motivic Integration on Toric Stacks}\label{motivic}

We consider motivic integration on a Deligne-Mumford stack as developed by Yasuda in \cite{YasMotivic}. We will compute motivic integrals associated to $T$-invariant divisors on $\mathcal{X}$ and show 
that they correspond to weighted $\delta$-vectors of an associated polyhedral complex. 

Recall that to any complex algebraic variety 
$X$ of dimension $r$, we can associate its \emph{Hodge polynomial} (see, for example, \cite{SriHodge})
\begin{equation*}
E_{X}(u,v) = \sum_{i,j = 0}^{r} (-1)^{i + j}h_{i,j}u^{i}v^{j} \in \mathbb{Z}[u,v].
\end{equation*}
The Hodge polynomial is determined by the properties
\begin{enumerate}
\item $h_{i,j} = \dim H^{j}(X, \Omega_{X}^{i})$ if $X$ is smooth and projective,
\item $E_{X}(u,v) = E_{U}(u,v) + E_{X \smallsetminus U}(u,v)$ if $U$ is an open subvariety of $X$,
\item $E_{X \times Y}(u,v) = E_{X}(u,v)E_{Y}(u,v)$.
\end{enumerate}
The second property means we can consider the Hodge polynomial of a constructible subset of a complex variety. 
For example, $E_{\mathbb{A}^{1}}(u,v) = E_{\mathbb{P}^{1}}(u,v) - E_{\{ \textrm{pt} \} }(u,v) = uv$ and hence 
$E_{\mathbb{A}^{n}}(u,v) = (uv)^{n}$. Similarly,  $E_{(\mathbb{C}^{*})^{n}}(u,v) = (uv - 1)^{n}$. More generally, we can compute the Hodge polynomial of 
any toric variety. Recall that if $\triangle$ is a fan in a lattice $N'_{\mathbb{R}}$, for some lattice $N'$, then its associated $h$-vector $h_{\triangle}(t)$ is defined by
\[
h_{\triangle}(t) = \sum_{\tau \in \triangle} t^{\dim \tau} (1 - t)^{\codim \tau}. 
\]

\begin{lemma}\label{kingston}
If $X = X(\triangle)$ is an $r$-dimensional toric variety associated to a fan $\triangle$, 
then 
\[
E_{X}(u , v) = (uv)^{r}h_{\triangle}((uv)^{-1}). 
\]
\end{lemma}
\begin{proof}
We can write $X$ as a disjoint union of torus orbits indexed by the cones of $\triangle$ (see, for example, \cite{FulIntroduction}). The orbit corresponding to $\tau$ is isomorphic to 
$(\mathbb{C}^{*})^{\codim \tau}$. We compute, using the example above,
\[
E_{X}(u,v) = \sum_{\tau \in \triangle} E_{(\mathbb{C}^{*})^{\codim \tau}}(u,v) = \sum_{\tau \in \triangle} (uv - 1)^{\codim \tau} =  (uv)^{r}h_{\triangle}((uv)^{-1}). 
\]
\end{proof}

\begin{rem}
If 
$\triangle$ is an $r$-dimensional, simplicial fan with convex support, then $X = X(\triangle)$ has no odd cohomology and the coefficient of $(uv)^{i}$ in 
$E_{X}(u,v)$ is equal to the dimension of the $2i^{\textrm{th}}$ cohomology group of $X$ with compact support and rational coefficients. 
This follows from the above lemma, Lemma 4.1 of \cite{YoWeightI} and
Poincar\'e duality.  
\end{rem}

Recall that we have projection morphisms
$\pi_{n}: |\mathcal{J}_{\infty}\mathcal{X}| \rightarrow |\mathcal{J}_{n}\mathcal{X}|$, 
for each non-negative integer $n$. A subset $A \subseteq |\mathcal{J}_{\infty}\mathcal{X}|$ is a \emph{cylinder} if 
$A = \pi_{n}^{-1}\pi_{n}(A)$ and $\pi_{n}(A)$ is a constructible subset, for some non-negative integer $n$. In this case, the measure
$\mu_{\mathcal{X}}(A)$ of $A$ is defined to be
\begin{equation*}
\mu_{\mathcal{X}}(A) = E_{\pi_{n}(A)}(u,v)(uv)^{-nd} \in \mathbb{Z}[u,u^{-1},v,v^{-1}].
\end{equation*}
By Lemma 3.18 of \cite{YasMotivic}, the right hand side is independent of the choice of $n$. 
The collection of cylinders is closed under taking finite unions and finite intersections and 
$\mu_{\mathcal{X}}$ defines a finite measure on  $|\mathcal{J}_{\infty}\mathcal{X}|$.

We will compute the measure of certain cylinders. Recall the decomposition of $|\mathcal{J}_{\infty}\mathcal{X}|'$ in Theorem \ref{decomp}.  
Every $w$ in $|\Sigma| \cap N$ can be uniquely written in the form 
$w = \{ w \} + \tilde{w}$, where $\{ w \}$ lies in $\BOX(\btau)$ for some $\tau \subseteq \sigma(w)$ and $\tilde{w}$ lies in $N_{\sigma(w)}$. Recall that 
$\sigma(w)$ is the cone of $\Sigma$ containing $w$ in
its relative interior and that $\psi: |\Sigma| \rightarrow \mathbb{R}$ is the piecewise $\mathbb{Q}$-linear function satisfying $\psi(b_{i}) = 1$
for $i = 1, \ldots, r$. 

\begin{lemma}\label{orbitz}
For any $w$ in $|\Sigma| \cap N$, the orbit $\tilde{\gamma}_{w} \cdot J_{\infty}(T)$ is a cylinder in $|\mathcal{J}_{\infty}\mathcal{X}|$ with measure
\begin{equation*}
\mu_{\mathcal{X}}(\tilde{\gamma}_{w} \cdot J_{\infty}(T)) = (uv - 1)^{d} (uv)^{-\psi(w) + \psi(\{w\}) - \dim \sigma(\{w\})    }.
\end{equation*}
\end{lemma}
\begin{proof}
Fix a $d$-dimensional cone $\sigma$ containing $w$ and consider the notation of (\ref{MickyO}).
The twisted arc $\tilde{\gamma}_{w}$ lies in 
$J_{\infty}^{(g)}\mathbb{A}^{d}/N(\sigma)$ and $\pi_{n}(\tilde{\gamma}_{w})$ lies in $J_{nl}^{(g)}\mathbb{A}^{d}/N(\sigma)$.
If we consider the natural projection
$\pi_{n}^{(g)}: J_{\infty}^{(g)}\mathbb{A}^{d} \rightarrow J_{nl}^{(g)}\mathbb{A}^{d}$,
then $\pi_{n}^{(g)}( (t^{lw_{1}}, \ldots, t^{lw_{d}}) \cdot J_{\infty}^{(l)}(T))$ is equal to 
\begin{equation*}
 \{ ( \sum_{k = \lambda_{1}}^{n_{1}} \alpha_{1, l(k + q_{1})} t^{l(k + q_{1})}, \ldots, 
 \sum_{k = \lambda_{d}}^{n_{d}} \alpha_{d, l(k + q_{d})} t^{l(k + q_{d})}  ) \mid \alpha_{i, w_{i}} \neq 0  \},
 \end{equation*}
for $n \geq \max \{ \lambda_{1} + 1, \ldots, \lambda_{d} + 1 \}$. Here $n_{i} = n$ if $q_{i} =0$ and $n_{i} = n -1$ if $q_{i} \neq 0$.
Note that $q_{i} \neq 0$ if and only if $\rho_{i} \subseteq \sigma(\{w\})$. Hence 
$(\pi_{n}^{(g)})^{-1}\pi_{n}^{(g)}( (t^{lw_{1}}, \ldots, t^{lw_{d}}) \cdot J_{\infty}^{(l)}(T)) = (t^{lw_{1}}, \ldots, t^{lw_{d}}) \cdot J_{\infty}^{(l)}(T)$,
and  $\pi_{n}^{(g)}( (t^{lw_{1}}, \ldots, t^{lw_{d}}) \cdot J_{\infty}^{(l)}(T))$ 
is isomorphic to 
\begin{equation}\label{cmon}
(\mathbb{C}^{*})^{d} \times \mathbb{A}^{       nd - \sum_{i = 1}^{d}\lambda_{i} - \dim \sigma(\{w\})              }. 
\end{equation}
Note that (\ref{cmon}) has a decomposition into a disjoint union of locally closed subspaces, each isomorphic to $(\mathbb{C}^{*})^{i}$ for some $i$, which is preserved after quotienting by the induced action of the finite group $N(\sigma)$. 
The result follows by considering the corresponding spaces and maps after quotienting by $N(\sigma)$, 
and observing that this doesn't affect the 
Hodge polynomial of (\ref{cmon}).
\end{proof}

Let $F: |\mathcal{J}_{\infty}\mathcal{X}|' \rightarrow \mathbb{Q}$ be a $J_{\infty}(T)$-invariant function on the space of twisted arcs of $\mathcal{X}$ and
let $A$ be a $J_{\infty}(T)$-invariant subset of  $|\mathcal{J}_{\infty}\mathcal{X}|$. By Theorem \ref{decomp}, we have a decomposition
\begin{equation*}
A \cap |\mathcal{J}_{\infty}\mathcal{X}|' =  \coprod_{\substack{ w \in |\Sigma| \cap N \\ \tilde{\gamma}_{w} \in A }} \tilde{\gamma}_{w} \cdot J_{\infty}(T).
\end{equation*}
We define 
\begin{equation*}
\int_{A} (uv)^{F} d\mu_{\mathcal{X}} := \sum_{\substack{ w \in |\Sigma| \cap N \\ \tilde{\gamma}_{w} \in A }} \mu_{\mathcal{X}}(\tilde{\gamma}_{w} \cdot J_{\infty}(T)) (uv)^{F(\tilde{\gamma}_{w})},
\end{equation*}
when the right hand sum is a well-defined element in $\mathbb{Z}[(uv)^{1/N}][[(uv)^{-1/N}]]$, for some positive integer $N$. With the definitions of Yasuda, this is the \emph{motivic integral} of $F$ over $A$.\footnote{To check the definitions 
agree, we can replace $A$ by $A \cap |\mathcal{J}_{\infty}\mathcal{X}|'$, by Propositions 3.24 and 3.25 in \cite{YasMotivic}. Now apply the decomposition of Theorem \ref{decomp}.}

Recall the involution $\iota: \BOX(\bSigma) \rightarrow \BOX(\bSigma)$ (Remark (\ref{ginvo})).
Following \cite{YasMotivic}, we define the \emph{shift function} $s_{\mathcal{X}}: |\mathcal{J}_{\infty}\mathcal{X}|' \rightarrow \mathbb{Q}$ by
\begin{equation*}
s_{\mathcal{X}}(\tilde{\gamma}_{w} \cdot J_{\infty}(T)) = \dim \sigma(\{w\}) - \psi(\{ w \}) = \psi(\iota(\{w\})). 
\end{equation*}
The shift function factors as 
$|\mathcal{J}_{\infty}\mathcal{X}|' \stackrel{\pi}{\rightarrow} |\mathcal{J}_{0}\mathcal{X}| = |\mathcal{I}\mathcal{X}| 
\stackrel{\sft}{\rightarrow} \mathbb{Q}$,
where the function $\sft$ is constant on the 
connected components of $|\mathcal{I} \mathcal{X}|$.
Borisov, Chen and Smith \cite{BCSOrbifold} showed that the connected components of $|\mathcal{I}\mathcal{X}|$ are indexed by $\BOX(\bSigma)$ (c.f. Remark \ref{doko}). 
The value of $\sft$ on the component of $|\mathcal{I} \mathcal{X}|$ 
corresponding to $g$ in $N(\sigma)$ is the age of $g^{-1}$ (see (\ref{age}) and Remark \ref{ginvo}). 
 
Consider a Kawamata log terminal pair $(\mathcal{X}, \mathcal{E})$, where $\mathcal{E}$ is a $T$-invariant $\mathbb{Q}$-divisor on $\mathcal{X}$ (see Section \ref{contact}). 
Following Yasuda (Definition 4.1 \cite{YasMotivic}), we consider the invariant $\Gamma(\mathcal{X}, \mathcal{E})$ defined by
\begin{equation*}
\Gamma(\mathcal{X}, \mathcal{E}) = \int_{|\mathcal{J}_{\infty}\mathcal{X}|} (uv)^{s_{\mathcal{X}} + \ord \mathcal{E}} d\mu_{\mathcal{X}}.
\end{equation*}
Yasuda showed that $\Gamma(\mathcal{X}, \mathcal{E})$ is a well-defined element in $\mathbb{Z}[(uv)^{1/N}][[(uv)^{-1/N}]]$, for some positive integer $N$ \cite[Proposition 4.15]{YasMotivic}.
Our goal is to compute such invariants and give them a combinatorial interpretation. We first recall the notion of weighted $\delta$-vector from \cite{YoWeightI}. 
Consider the polyhedral complex
$Q = \{ v \in |\Sigma| \, | \, \psi(v) \leq 1 \}$
and let  $\lambda: |\Sigma| \rightarrow \mathbb{R}$ be a piecewise $\mathbb{Q}$-linear function satisfying $\lambda(b_{i}) > - 1$ for $i = 1, \ldots, r$.
The \emph{weighted $\delta$-vector} $\delta^{\lambda}(t)$ is defined by
\[
\delta^{\lambda}(t) = (1  - t)^{d + 1}(1 +  \sum_{m \geq 1} \sum_{v \in mQ \cap N} t^{\psi(v) - \lceil \psi(v) \rceil + \lambda(v) + m}).
\]
The following symmetry property was established in Corollary 2.13 of \cite{YoWeightI}.
\begin{theorem}[\cite{YoWeightI}]\label{stu}
With the notation above, the expression $\delta^{\lambda}(t)$ is a rational function in $\mathbb{Q}(t^{1/N})$, for some positive integer $N$.  
If $\Sigma$ is a complete fan, 
then 
\begin{equation*}
\delta^{\lambda}(t) = t^{d} \delta^{\lambda}(t^{-1}).
\end{equation*}
\end{theorem}


For each cone $\tau$ in $\Sigma$, let $N_{\tau}$ be the sublattice of $N$ generated by $\{ b_{i} \mid \rho_{i} \subseteq \tau \}$ and let $\Sigma_{\tau}$ be the simplicial fan in $(N/N_{\tau})_{\mathbb{R}}$ with cones given by the projections 
of the cones in $\Sigma$ containing $\tau$. 
We define $h_{\tau}^{\lambda}(t)$ to be the expression
\begin{equation*}
h_{\tau}^{\lambda}(t) =  \sum_{\tau \subseteq \sigma}t^{\sum_{\rho_{i} \subseteq \sigma \backslash \tau} \lambda(b_{i})}  t^{\dim \sigma - \dim \tau}(1 - t)^{\codim \sigma} 
 \prod_{\rho_{i} \subseteq \sigma \backslash \tau} (1 - t)/(1 - t^{\lambda(b_{i})+ 1}),
\end{equation*}
so that when $\lambda \equiv 0$, $h_{\tau}^{\lambda}(t)$ is the $h$-vector of $\Sigma_{\tau}$ (see, for example, \cite{StaEnumerative}). 
By expanding and collecting terms, we have  
\begin{equation}\label{ditch}
t^{\codim \tau} h_{\tau}^{\lambda}(t^{-1}) = 
(t - 1)^{\codim \tau} \sum_{\tau \subseteq \sigma}  \prod_{\rho_{i} \subseteq \sigma \backslash \tau} 1 /(t^{\lambda(b_{i})+ 1} - 1).
\end{equation}
The following expression for the weighted $\delta$-vector was 
established in Proposition 2.7 of \cite{YoWeightI},
\begin{equation}\label{mccoy}
\delta^{\lambda}(t)  = \sum_{\tau \in \Sigma} h_{\tau}^{\lambda}(t) \sum_{ v \in \BOX(\btau)} t^{\psi(v) + \lambda(v)}
\prod_{\rho_{i} \subseteq \tau} (t -1)/(t^{\lambda(b_{i}) + 1} -1) .
\end{equation}
After rearranging we see that 
\begin{equation}\label{blue}
\delta^{\lambda}(t^{-1}) = \sum_{\tau \in \Sigma} h_{\tau}^{\lambda}(t^{-1}) \sum_{ v \in \BOX(\btau)} t^{- \psi(v) - \lambda(v) + \sum_{\rho_{i} \subseteq \tau}\lambda(b_{i}) }
\prod_{\rho_{i} \subseteq \tau} (t -1)/(t^{\lambda(b_{i}) + 1} -1) .
\end{equation}



\begin{theorem}\label{tropo}
If $(\mathcal{X}, \mathcal{E})$ is a Kawamata log terminal pair, where $\mathcal{E}$ is a $T$-invariant $\mathbb{Q}$-divisor on $\mathcal{X}$, 
then the corresponding piecewise $\mathbb{Q}$-linear function $\lambda = \psi_{q_{*}\mathcal{E}}$
satisfies $\lambda(b_{i}) > - 1$ for $i = 1, \ldots, r$, and
\begin{equation*}
\Gamma(\mathcal{X}, \mathcal{E}) = (uv)^{d} \delta^{\lambda}((uv)^{-1}).
\end{equation*}
In particular, $\Gamma(\mathcal{X}, \mathcal{E})$  is a rational function in $\mathbb{Q}(t^{1/N})$, for some positive integer $N$. 
If $\Sigma$ is a complete fan, then
\begin{equation*}
\Gamma(\mathcal{X}, \mathcal{E})(u,v) = (uv)^{d}  \Gamma(\mathcal{X}, \mathcal{E})(u^{-1}, v^{-1}) =  \delta^{\lambda}(uv).
\end{equation*}
Moreover, every weighted $\delta$-vector has the form $\delta^{\lambda}(uv) = (uv)^{d}\Gamma(\mathcal{X}, \mathcal{E})(u^{-1}, v^{-1})$, for some such
pair $(\mathcal{X}, \mathcal{E})$.
\end{theorem}
\begin{proof}
We showed in Lemma \ref{order} that 
$\ord \mathcal{E}( \tilde{\gamma}_{w} \cdot J_{\infty}(T) ) = -\lambda(w)$ and
hence 
$(s_{\mathcal{X}} + \ord \mathcal{E})( \tilde{\gamma}_{w} \cdot J_{\infty}(T)) =  \dim \sigma(\{w\}) - \psi(\{ w \}) -\lambda(w)$.
Using Lemma \ref{orbitz}, we compute
\[
\Gamma(\mathcal{X}, \mathcal{E}) = (uv - 1)^{d} \sum_{w \in |\Sigma| \cap N} (uv)^{-\psi(w) -\lambda(w) }.
\]
For each $w$ in $|\Sigma| \cap N$, we have a unique decomposition
\begin{equation*}
w = \{w\} + w' + \sum_{\rho_{i} \subseteq \sigma(w) \smallsetminus \sigma(\{w\})} b_{i},
\end{equation*}
where $w'$ is a non-negative linear combination of $\{ b_{i} \mid \rho_{i} \subseteq \sigma(w) \}$. 
Here $\sigma(w)$ is the cone containing $w$ in its relative interior. 
Using this decomposition, 
we compute the following expression for $\Gamma(\mathcal{X}, \mathcal{E})$,
\[
\sum_{\substack{ \tau \in \Sigma \\    v \in \BOX(\btau)            } } (uv - 1)^{d}(uv)^{-(\psi + \lambda)(v)}
\sum_{\tau \subseteq \sigma} (uv)^{- \sum_{\rho_{i} \subseteq \sigma \smallsetminus \tau} (\lambda(b_{i}) + 1)}
\prod_{\rho_{i} \subseteq \sigma} 1/( 1 - (uv)^{-\lambda(b_{i}) - 1}  ).
\]
Rearranging and using (\ref{ditch}) gives 
\begin{align*}
\Gamma(\mathcal{X}, \mathcal{E})& = \sum_{\tau \in \Sigma} (uv)^{\codim \tau}h_{\tau}^{\lambda}((uv)^{-1})  \times \\
&\sum_{ v \in \BOX(\btau)} (uv)^{\sum_{\rho_{i} \subseteq \tau}\lambda(b_{i}) + \dim \tau - \psi(v) -   \lambda(v)} 
\prod_{\rho_{i} \subseteq \tau} (uv -1)/((uv)^{\lambda(b_{i}) + 1} -1). 
\end{align*}
Comparing with (\ref{blue}) gives $\Gamma(\mathcal{X}, \mathcal{E}) = (uv)^{d} \delta^{\lambda}((uv)^{-1})$. The second statement follows from Theorem \ref{stu}. 
If  $\lambda: |\Sigma| \rightarrow \mathbb{R}$ is a piecewise $\mathbb{Q}$-linear function satisfying $\lambda(b_{i}) > - 1$ for $i = 1, \ldots, r$, then we may consider the 
corresponding $T$-invariant $\mathbb{Q}$-divisor $E$ on $X$.  The pair $(\mathcal{X}, q^{*}E)$ is a Kawamata log terminal pair and, by the above argument, 
$\delta^{\lambda}(uv) = (uv)^{d}  \Gamma(\mathcal{X}, q^{*}E)(u^{-1}, v^{-1})$. Hence every weighted $\delta$-vector corresponds to a motivic integral on $\mathcal{X}$. 
\end{proof}

\begin{rem}
By Theorem \ref{tropo} and (\ref{blue}), we obtain a formula for  the invariant $\Gamma(\mathcal{X}, \mathcal{E})$. This formula is a stacky analogue of
Batyrev's formula for the motivic integral of a simple normal crossing divisor on a smooth variety (see \cite{VeyArc}, \cite{BatNon}).
\end{rem}

\begin{rem}
When $\mathcal{E} = 0$, $\Gamma(\mathcal{X}, 0)$ is a polynomial in $uv$ of degree $d$ and the coefficient of $(uv)^{j}$
is equal to the dimension of the $2j^{\textrm{th}}$ orbifold cohomology
group of $\mathcal{X}$ with compact support \cite{YasTwisted}.
When $\Sigma$ is complete, the symmetry 
$\Gamma(\mathcal{X}, 0)(u,v) = (uv)^{d}  \Gamma(\mathcal{X}, 0)(u^{-1}, v^{-1})$ is a consequence of Poincar\'e duality for orbifold cohomology \cite{CRNew}. 
\end{rem}

\section{The Transformation Rule}\label{transform}

A morphism $f: \mathcal{Y} \rightarrow \mathcal{X}$ of smooth Deligne-Mumford stacks is \emph{birational} if there exist
open, dense substacks $\mathcal{Y}_{0}$ of $\mathcal{Y}$ and $\mathcal{X}_{0}$ of $\mathcal{X}$ such that $f$ induces an isomorphism 
$\mathcal{Y}_{0} \cong \mathcal{X}_{0}$ (Definition 3.36 \cite{YasMotivic}). If $f$ is proper and birational then Yasuda \cite{YasMotivic} proved a transformation rule 
relating motivic integrals on $\mathcal{X}$ 
to motivic integrals
on $\mathcal{Y}$, generalising a classic result of Kontsevich for smooth varieties (see, for example, \cite{VeyArc}).  The goal of this section is to interpret the transformation rule in our context in order to give a geometric proof of a combinatorial result in \cite{YoWeightI}. 

If $\btriangle = (N, \triangle, \{ \bar{b}_{i} \})$ and $\bSigma = (N, \Sigma, \{ b_{i} \} )$ are stacky fans, we say that $\btriangle$ \emph{refines} $\bSigma$
if 
\begin{enumerate}
\item The fan $\triangle$ refines $\Sigma$ in $N_{\mathbb{R}}$, 
\item\label{hoy}  For any ray $\bar{\rho}_{i}$ in $\triangle$, $\bar{b}_{i}$ is an integer combination of the lattice points $\{ b_{j} \mid \rho_{j} \subseteq \sigma \}$, where $\sigma$ is any cone of 
$\Sigma$ containing $\bar{\rho}_{i}$. 
\end{enumerate}
If $\btriangle$ \emph{refines} $\bSigma$, we have an induced morphism of toric stacks $f: \mathcal{X}(\btriangle) \rightarrow \mathcal{X}(\bSigma)$ 
(Remark 4.5 \cite{BCSOrbifold}) such that $f$ restricts to the identity map on the torus and hence is birational. Note that the corresponding map of coarse moduli spaces is proper \cite{FulIntroduction}. 
Since the morphism from a Deligne-Mumford stack to its coarse moduli space is proper \cite{KMQuotients}, it follows that $f$ is proper \cite{LMBChamps}.
We will give a local description of $f$.
Let $\tau$ be a maximal cone in 
$\triangle$, contained in a maximal cone $\sigma$ of $\Sigma$. By Property  (\ref{hoy}),  we have an inclusion of lattices
$N_{\tau} \hookrightarrow N_{\sigma}$,
inducing a homomorphism of groups
$j: N(\tau) \rightarrow N(\sigma)$.
Let $\tau'$ be the cone generated by $\{ \bar{b}_{i} \mid \bar{\rho}_{i} \subseteq \tau \}$ in $N_{\tau}$ and 
let $\sigma'$ be the cone generated by $\{ b_{i} \mid \rho_{i} \subseteq \sigma \}$ in $N_{\sigma}$. The inclusion
$\tau' \cap N_{\tau} \hookrightarrow \sigma' \cap N_{\sigma}$
induces a $j$-equivariant map
$\phi: \Spec \mathbb{C}[ \check{\tau}' \cap M_{\tau}   ]  \rightarrow \Spec \mathbb{C}[ (\sigma')^{\vee} \cap M_{\sigma}   ]$.  
Taking stacky quotients of both sides yields the restriction of $f$ to $\mathcal{X}(\btau)$,
$f: \mathcal{X}(\btau) = [ \Spec \mathbb{C}[ \check{\tau}' \cap M_{\tau}   ] / N(\tau) ] \rightarrow 
\mathcal{X}(\bsigma) = [ \Spec \mathbb{C}[ (\sigma')^{\vee} \cap M_{\sigma}   ] / N(\sigma) ]$.

Given a morphism $g: \mathcal{Y} \rightarrow \mathcal{Z}$ of Deligne-Mumford stacks, Yasuda described a natural morphism
$g_{\infty}: |\mathcal{J}_{\infty}\mathcal{Y}| \rightarrow |\mathcal{J}_{\infty}\mathcal{Z}|$ (Proposition 2.14 \cite{YasMotivic}). We outline his construction. 
Given a representable morphism
$\gamma: \mathcal{D}^{l}_{\infty, \mathbb{C}} \rightarrow \mathcal{Y}$,
we can consider the composition
$\gamma': \mathcal{D}^{l}_{\infty, \mathbb{C}} \rightarrow \mathcal{Y} \rightarrow \mathcal{X}$.
If $l'$ is a positive integer dividing $l$, then we have a group homomorphism
$\mu_{l} \rightarrow \mu_{l'}$,
$\zeta_{l} \mapsto (\zeta_{l})^{l/l'} = \zeta_{l'}$,
and an equivariant map
$\Spec \mathbb{C}[[t]] \rightarrow \Spec \mathbb{C}[[t]]$, 
$t \mapsto t^{l/l'}$.
Taking stacky quotients of both sides gives a morphism
$\mathcal{D}^{l}_{\infty, \mathbb{C}} \rightarrow \mathcal{D}^{l'}_{\infty, \mathbb{C}}$.
By Lemma 2.15 of \cite{YasMotivic}, there exists a unique positive integer $l'$ dividing $l$ so that $\gamma'$ factors (up to $2$-isomorphism) as 
$\mathcal{D}^{l}_{\infty, \mathbb{C}} \rightarrow \mathcal{D}^{l'}_{\infty, \mathbb{C}} \stackrel{\psi}{\rightarrow} \mathcal{X}$,
where $\psi$ is representable, and then $g_{\infty}(\gamma) = \psi$. 

Let $Y$ be the coarse moduli space of $\mathcal{Y}$ and let $Z$ be the coarse moduli space of $\mathcal{Z}$. Then  
$g: \mathcal{Y} \rightarrow \mathcal{Z}$ induces a morphism $g': Y \rightarrow Z$. 
Note that, by considering coarse moduli spaces, the morphism 
$\mathcal{D}^{l}_{\infty, \mathbb{C}} \rightarrow \mathcal{D}^{l'}_{\infty, \mathbb{C}}$ yields the identity morphism on $\Spec \mathbb{C}[[t]]$. 
It follows that we have a commutative diagram
\begin{equation}\label{Kirk}
\xymatrix{ |\mathcal{J}_{\infty}\mathcal{Y}|  \ar[r]^{g_{\infty}} \ar[d]^{\tilde{\pi}_{\infty} } & 
|\mathcal{J}_{\infty}\mathcal{Z}| \ar[d]^{  \tilde{\pi}_{\infty}   } \\
J_{\infty}(Y) \ar[r]^{g_{\infty}'}  & J_{\infty}(Z) .
}  
\end{equation}

Consider the notation of Theorem \ref{decomp}.    

\begin{lemma}\label{light}
If $\btriangle = (N, \triangle, \{ \bar{b}_{i} \})$ is a stacky fan refining $\bSigma = (N, \Sigma, \{ b_{i} \})$, then the birational morphism 
$f: \mathcal{X}(\btriangle) \rightarrow \mathcal{X}(\bSigma)$
induces a $J_{\infty}(T)$-equivariant map
\begin{equation*}
f_{\infty}: |\mathcal{J}_{\infty}\mathcal{X}(\btriangle)|' \rightarrow |\mathcal{J}_{\infty}\mathcal{X}(\bSigma)|'
\end{equation*}
\begin{equation*}
f_{\infty}(\tilde{\gamma}_{v}) = \tilde{\gamma}_{v}.
\end{equation*}
\end{lemma}
\begin{proof}
By considering coarse moduli spaces, $f: \mathcal{X}(\btriangle) \rightarrow \mathcal{X}(\bSigma)$ gives rise to the toric morphism 
$f': X(\triangle) \rightarrow X(\Sigma)$,
with induced map of arc spaces
$f_{\infty}': J_{\infty}(X(\triangle))'  \rightarrow J_{\infty}(X(\Sigma))'$.
Consider the commutative diagram (\ref{Kirk}). 
By Theorem \ref{decomp}, the maps $\tilde{\pi}_{\infty}$ are $J_{\infty}(T)$-equivariant bijections satisfying 
$\tilde{\pi}_{\infty}(\tilde{\gamma}_{v}) = \gamma_{v}$. Hence we only need to show that $f_{\infty}'$ is $J_{\infty}(T)$-equivariant
and satisfies $f_{\infty}'(\gamma_{v}) = \gamma_{v}$. This fact is well-known  but we recall a proof for the convenience of the reader. 
Suppose $v$ lies in a cone $\tau$ of $\triangle$ and let $\sigma$ be a cone in $\Sigma$ containing $\tau$. 
The arc $\gamma_{v}$ corresponds to the ring homomorphism
$\mathbb{C}[\check{\tau} \cap M] \rightarrow \mathbb{C}[[t]]$,
$\chi^{u} \mapsto t^{\langle u, v \rangle}$
and hence $f_{\infty}'(\gamma_{v})$ corresponds to the ring homomorphism
$\mathbb{C}[\sigma^{\vee} \cap M] \hookrightarrow \mathbb{C}[\check{\tau} \cap M] \rightarrow \mathbb{C}[[t]]$,
$\chi^{u} \mapsto \chi^{u} \mapsto t^{\langle u, v \rangle}$.
We see that  $f_{\infty}'$ is $J_{\infty}(T)$-equivariant
and $f_{\infty}'(\gamma_{v}) = \gamma_{v}$. 
\end{proof}

We now state Yasuda's transformation rule in the case of toric stacks.

\begin{theorem}[The Transformation Rule, \cite{YasMotivic}]\label{yap}
Let $\btriangle = (N, \triangle, \{ \bar{b}_{i} \})$ be a stacky fan refining $\bSigma = (N, \Sigma, \{ b_{i} \})$, with corresponding birational morphism 
$f: \mathcal{X}(\btriangle) \rightarrow \mathcal{X}(\bSigma)$.
If $F: |\mathcal{J}_{\infty}\mathcal{X}(\bSigma)|' \rightarrow \mathbb{Q}$ is a $J_{\infty}(T)$-invariant function and
$A$ is a $J_{\infty}(T)$-invariant subset of  
$|\mathcal{J}_{\infty}\mathcal{X}(\bSigma)|'$, 
then
\begin{equation*}
\int_{A} (uv)^{   s_{\mathcal{X}(\Sigma)} + F    } d\mu_{\mathcal{X}} 
= \int_{f_{\infty}^{-1}(A)} (uv)^{s_{\mathcal{X}(\triangle)}  + F \circ f_{\infty} 
- \ord K_{\mathcal{X}(\triangle)/\mathcal{X}(\Sigma)}                  } d\mu_{\mathcal{X}'},
\end{equation*}
where $K_{\mathcal{X}(\triangle)/\mathcal{X}(\Sigma)} = K_{\mathcal{X}(\triangle)} - f^{*}K_{\mathcal{X}(\Sigma)}$.
\end{theorem}

As a corollary, we deduce a geometric proof of the following result, which was proved using combinatorial methods in 
Proposition 2.13 of \cite{YoWeightI}.
If $\bSigma = (N, \Sigma, \{b_{i}\})$ is a stacky fan and $\lambda$ is a piecewise $\mathbb{Q}$-linear function on $|\Sigma|$, we fix the following notation. 
Let $\psi_{\Sigma} = \psi_{q_{*}K_{\mathcal{X}(\Sigma) } }$ be the piecewise $\mathbb{Q}$-linear function on $|\Sigma|$ satisfying $\psi_{\Sigma}(b_{i}) = 1$, 
and let   $\delta^{\lambda}_{\Sigma}(t)$ denote the weighted $\delta$-vector corresponding to $\lambda$.


\begin{cor}\label{fastball}
Let $N$ be a lattice of rank $d$ and let $\bSigma = (N, \Sigma, \{ b_{i} \})$ and $\btriangle = (N , \triangle, \{ b'_{j} \})$ be stacky fans such that $|\Sigma| = |\triangle|$. Let $\lambda$ be
a piecewise $\mathbb{Q}$-linear function with respect to $\Sigma$ satisfying $\lambda(b_{i}) > -1$ for every $b_{i}$, and  set 
$\lambda' = \lambda +  \psi_{\Sigma} - \psi_{\triangle}$.
If $\lambda'$ is piecewise $\mathbb{Q}$-linear with respect to $\triangle$ and satisfies $\lambda'(b'_{j}) > -1$ for every $b'_{j}$,
then 
$\delta^{\lambda}_{\Sigma}(t) = \delta^{\lambda'}_{\triangle}(t)$.
\end{cor}
\begin{proof}
Let $\tilde{\Sigma}$ be a common refinement of $\Sigma$ and $\triangle$. We can choose lattice points $\{\tilde{b}_{i}\}$ on the rays of $\tilde{\Sigma}$ so that 
$\tilde{\bSigma} = (N, \tilde{\Sigma}, \{\tilde{b}_{i}'\})$ is a stacky fan refining $\bSigma$ and $\btriangle$. Hence we can reduce to the case when $\btriangle$ refines $\bSigma$. 
In this case, 
consider the corresponding birational morphism 
$f: \mathcal{X}(\btriangle) \rightarrow \mathcal{X}(\bSigma)$.
By Lemma \ref{order}, there is a Kawamata log terminal pair $(\mathcal{X}(\bSigma), \mathcal{E})$ such that 
$\ord \mathcal{E}( \tilde{\gamma}_{w} \cdot J_{\infty}(T) ) = -\lambda(w)$.
It follows from Lemma \ref{light}  that 
$(\ord \mathcal{E} \circ f_{\infty} )( \tilde{\gamma}_{w} \cdot J_{\infty}(T) ) = -\lambda(w)$ and
Lemma \ref{order} and Lemma \ref{light} imply that
$\ord K_{\mathcal{X}(\triangle)/\mathcal{X}(\Sigma)}   ( \tilde{\gamma}_{w} \cdot J_{\infty}(T) ) = \psi_{\Sigma}(w)  - \psi_{\triangle}(w)$.
The result now follows from Theorem \ref{tropo} and Theorem \ref{yap}, with $F = \ord \mathcal{E}$ and $A = |\mathcal{J}_{\infty}\mathcal{X}(\bSigma)|$.
\end{proof}

\section{Remarks}\label{remedy}

More generally, we can replace $N$ by a finitely generated abelian group of rank $d$. All the results go through with minor modifications. 
We mention the local construction of the toric stack \cite{BCSOrbifold}.
Let $\bar{N}$ be the lattice given by the image of $N$ in 
$N_{\mathbb{R}}$ and  for each $v$ in $N$, let $\bar{v}$ denote the image of $v$ in $\bar{N}$.
Let $\Sigma$ be a complete, simplicial, rational fan in $N_{\mathbb{R}}$ and 
let  $\bar{v}_{1}, \ldots, \bar{v}_{r}$ be the primitive integer generators of $\Sigma$ in $\bar{N}$. 
Fix elements $b_{1}, \ldots , b_{r}$ in $N$ such that $\bar{b}_{i} = a_{i}\bar{v}_{i}$, for some positive integer $a_{i}$.
The data $\bSigma = (N, \Sigma, \{ b_{i} \})$ is a \emph{stacky fan}. 
For each maximal cone $\sigma$ of $\Sigma$, let $N_{\sigma}$ denote the subgroup of $N$ generated by $\{ b_{i} \, | \, \rho_{i} 
\subseteq \sigma \}$. We obtain a homomorphism of finite groups
$N(\sigma) = N/N_{\sigma} \rightarrow \bar{N}(\sigma) = \bar{N}/\bar{N}_{\sigma}$.
Composing with the injection 
$\bar{N}(\sigma) \rightarrow (\mathbb{C}^{*})^{d}$ from Section \ref{tstack} gives a homomorphism
$N(\sigma) \rightarrow \bar{N}(\sigma) \rightarrow (\mathbb{C}^{*})^{d}$, and 
$\mathcal{X}(\bsigma) = [\mathbb{A}^{d}/ N(\sigma)]$.

Using this more general setup, one can apply Theorem \ref{decomp} to the $T$-invariant closed substacks of $\mathcal{X}$ to give a decomposition of $|\mathcal{J}_{\infty}\mathcal{X}|$
into $J_{\infty}(T)$-orbits.

\bibliographystyle{amsplain}
\bibliography{alan}

\end{document}